%% file: draft.tex
\documentclass{article}
\usepackage{mystyle}

\begin{document} 
\title{\Large Trace and flux a priori error estimates in finite element approximations of Signorni-type problems}%

\author{
{\sc
Olaf Steinbach\thanks{Email: o.steinbach@tugraz.at}
 }\\[2pt]
 Institut f\"ur Numerische Mathematik, TU Graz, Steyrergasse 30,\\ 8010 Graz, Austria\\[6pt]
{\sc and}\\[6pt]
{\sc Barbara Wohlmuth\thanks{Email: wohlmuth@ma.tum.de}  and
Linus Wunderlich\thanks{Corresponding author. Email: linus.wunderlich@ma.tum.de}
}\\[2pt]  M2 -- Zentrum Mathematik, Technische Universit\"at M\"unchen,\\ Boltzmannstra\ss{}e 3, 85748 Garching, Germany
}

\maketitle

\begin{abstract}\noindent
Variational inequalities play in many applications an important role and are an active research area. Optimal a priori error estimates in the natural energy norm do exist but only very few results in other norms exist. Here we consider as prototype a simple Signorini problem and provide new optimal order a priori error estimates for the trace and the flux on the Signorini boundary. The a priori analysis is based on the exact and a mesh-dependent Steklov--Poincar\'e operator as well as on duality in Aubin--Nitsche type arguments. Numerical results illustrate the convergence rates of the finite element approach.
\end{abstract}

%\begin{AMS}
%35J85, 65N30
%\end{AMS}

\section{ Introduction}
\input{01_introduction}

\section{ Problem setting and main result}
\input{02_problem}

\section{Equivalent reformulations}
\input{03_reformulations}

\section{Application of a Strang lemma}
\input{04_strang}

\section{A priori estimate of the primal trace}
\input{05_boundary_estimates}

\section{Lagrange multiplier estimates}
\input{06_LM_approx}
 
\section{Numerical results}  
\input{07_numerical_results}

\section{Conclusion}
\input{08_conclusion}

\section*{Funding}
Support from the International Research Training Group IGDK 1754, funded by the German Research Foundation (DFG) and the Austrian Science Fund (FWF), is gratefully acknowledged.
The second and third authors have been supported by the German Research Foundation (DFG) in Project WO 671/15-1 and the Priority Programme “Reliable Simulation Techniques in Solid Mechanics. Development of Non-standard Discretisation Methods, Mechanical and Mathematical Analysis” (SPP 1748).

\bibliographystyle{alpha}
\bibliography{./bib/bibliography}

\end{document}

%% file: 01_introduction.tex
Signorini-type problems are nonlinear boundary value problems that can be regarded as a simplified scalar model of elastic contact problems which are of interest in many engineering applications, see~\cite{laursen:02,wriggers:02}.
Signorini and contact problems share a similar formulation and their
approximation remains a challenging task due to the nonlinear boundary
condition.
A priori error estimates in the $H^1(\Omega)$ norm for such problems were investigated over many years, see~\cite{scarpini:77,  belgacem:03} for Signorini and~\cite{belgacem:99, sassi:99, hild:00} for contact problems.
Optimal a priori error estimates for two body contact problems in the $H^1(\Omega)$ norm
were established in~\cite{wohlmuth:05} and more recently reconsidered in \cite{hild:12, hild:1X}. However the optimal order a priori analysis for different norms of interest is still missing.

In this work, we restrict ourselves to the Poisson equation with unilateral Signorini
boundary conditions and provide  optimal order convergence rates in  norms associated with
the Signorini boundary $\Gamma_S$.
More precisely, we consider a priori error estimates for the trace in the
$H_{00}^{1/2}(\Gamma_S)$ norm and for the Lagrange multiplier, i.e., the flux, in the
$H^{-1/2}(\Gamma_S)$ norm. 
As a corollary we show improved a priori estimates in the $L^2$ norm for the primal variable on $\Omega$ and for the dual variable on $\Gamma_S$. 
While convergence rates for traces can often be established using estimates in the domain, these rates are typically not optimal.
To the best of our knowledge, no optimal order error estimates  in different norms than $H^1(\Omega)$ have been so far proven.
The order of the finite element approximation in the $L^2(\Omega)$ norm is firstly addressed in the early paper by~\cite{natterer:76}. However, the theoretical results are limited to very special situations. A generalization can be found in~\cite{hild:01, suttmeier:08}, but for a straightforward application to Signorini problems, the required dual regularity is lacking, so we do not follow these ideas.
Recently introduced techniques allow optimal estimates on interfaces and
boundaries for linear problems  under moderately stronger regularity assumptions, see~\cite{apel:12b, wohlmuth:12, human:13, waluga:13, apel:14, larson:14}.
These techniques can also be used to compensate a lack of regularity in the dual problem, see~\cite{horger:14}. 
 A reformulation of the primal variational inequality on the boundary, as
applied in~\cite{spann:93, steinbach:99, steinbach:12}, and a Strang lemma for
variational inequalities allow us to  use these techniques for the
nonlinear Signorini problem. 

This article is structured as follows:
In the next section, we state the Signorini-type problem and its discretization
as a primal formulation. In Section 3, two reformulations which play an important part in the analysis are briefly recalled; namely a saddle point problem and  a variational formulation of the Schur complement. Since the Galerkin formulation of the continuous Schur complement differs from the discrete Schur complement, a Strang lemma is applied in Section 4, and the error is related to the difference  of a Steklov--Poincar\'e operator and a finite element approximation.   
In Section 5, a rate for the primal  error in the $H^{1/2}_{00}(\Gamma_S)$ norm  is proven based on  anisotropic norms and dual problems with local data. As a corollary improved rates for the $L^2(\Omega)$ norm are shown.
The results are extended in Section 6, where optimal rates for the Lagrange multiplier in the natural $H^{-1/2}(\Gamma_S)$ 
and also in the stronger $L^2(\Gamma_S)$ norm are derived. Finally in Section 7, numerical results are
presented which confirm the new theoretical a priori bounds and illustrate some
additional aspects.

%% file: 02_problem.tex
We consider the Poisson equation with Signorini-type boundary conditions. The partial differential equation is defined in a domain $\Omega \subset \mathbb R^d$, $d = 2,3$. We assume $\Omega$ to be polyhedral, convex, and bounded. The boundary $\Gamma \coloneq \partial \Omega$ is divided into two disjoint open parts $\Gamma = \overline \Gamma_D \cup \overline \Gamma_S$, such that $\Gamma_D$ has a positive Lebesgue measure. For simplicity of notation, we assume $\Gamma_S$ to be one facet of the boundary $\Gamma$. For $f\in L^2(\Omega)$, $g\in H^{1/2}(\Gamma_S)$,
 we consider Equations~\eqref{eq:strong_formulation01} to~\eqref{eq:strong_formulation03}:
\begin{subequations}\label{eq:strong_formulation}
\begin{align}
	-\Delta u &= f \quad \text{ in } \Omega, \label{eq:strong_formulation01} \\
	u &= 0 \quad \text{ on } \Gamma_D, \\
	\partial_n u \leq 0, \quad u \leq g, & \quad (u-g) ~ \partial_n u = 0 \quad \text{ on } \Gamma_S \label{eq:strong_formulation03}.
\end{align}
\end{subequations}
The problem can be regarded as a simplified contact problem where the constraints on $\Gamma_S$ play the role of a nonpenetration condition. 

  %\marginpar{Eigentlich staerker als notwendig (muss nur noch vom Gitter aufgeloest werden... trotzdem ok?}
%and $g$ to be affine. Hence, we have to assume $g\geq 0$. % Without loss of generality $\Gamma_S$ lies in the $x_1 \ldots x_{d-1}$ -plane and $\Omega$ lies in the positive half-space $\{(x_1, \dotsc, x_{d-1}, \tau)\colon \tau \geq 0\}$. % We assume $g$ to be affine.
 The actual contact set $\Gammaact \coloneq \{x \in \Gamma_S\colon u(x) = g(x) \}$ is assumed to be a compact subset of $\Gamma_S$. %\marginpar{ Kann man das in der starken Formulierung schon fordern der erst später?}
 With regards to the Dirichlet condition, we note that $g$ has to be positive in a neighbourhood of $\partial \Gamma_S$.

\begin{remark}\label{remark_regularity}
	In general, weak solutions of  Dirichlet--Neumann problems with smooth data can be represented as a series of singular components and a smooth part. The first singular component has typically a regularity of $H^{3/2-\varepsilon}(\Omega)$. However due to the sign-condition of the Signorini boundary, the regularity is improved. As long as no jump of the outer unit normal is present at the boundary of $\Gammaact \subset \Gamma_S$,  the stress intensity factor associated with the first singular component has to be zero. We are interested in the effects of the approximation caused by the Signorini boundary condition, so let us assume, that these singular parts do neither appear at any other part of the boundary. 
	Hence we assume $f$ to be sufficiently smooth and the solution to be $H^{5/2-\varepsilon}(\Omega)$ regular, see~\cite{moussaoui:92}.
\end{remark}

\subsection{Weak formulations}
The nonlinear Signorini boundary condition yields a constrained minimization problem as the weak formulation, e.g.,~\cite{glowinski:84, kinderlehrer:00}.
Let $ V  \coloneq \{v\in H^1(\Omega)\colon \left. v \right|_{\Gamma_D} = 0\}$ and denote the trace space of $V$ restricted to $\Gamma_S$ as $W \coloneq H_{00}^{1/2}(\Gamma_S)$.
For simplicity of notation, we omit the trace operator whenever there is no ambiguity.
 We define the convex set of admissible functions by $K \coloneq \{v\in V\colon \left. v \right|_{\Gamma_S} \leq g\}$, the bilinear form $a(u,v) \coloneq \int_\Omega \nabla u^{\rm T} \nabla v~\mathrm{d}x$ and the linear form $f(v) \coloneq \int_\Omega f v~\mathrm{d}x$.

The weak solution $u \in K$ then satisfies the variational inequality
	\begin{align}\label{vi_cont}
		a(u, v-u) \geq f( v -u), \quad v \in K.
	\end{align}

For the discretization, we assume a family of shape-regular simplicial triangulations $\mathcal{T}_h$. We denote by $N_{V_h}$ the number of vertices of the triangulation except the ones on $\overline{\Gamma}_D$ and by $N_{M_h}$ the number of vertices on $\Gamma_S$.  Note that, since the Signorini boundary is a facet of the polyhedral domain, both boundary parts are exactly represented by the triangulation. %\marginpar{ kommasetzung und grammatik...}
We define the discrete primal space using first order conforming finite elements $V_h \coloneq \{v_h\in C(\overline{\Omega})\colon \left. v_h \right|_{T} \in \mathbb{P}_1(T),  T\in \mathcal{T}_h, \left. \Tr v_h \right|_{\Gamma_D} = 0  \}$, spanned by the nodal Lagrange basis $\varphi_i,$ $i =1 , \ldots, N_{V_h}$ and denote the discrete trace space restricted to $\Gamma_S$ by $W_h$. Let $g_h \in W_h$ denote a suitable approximation of $g$.	
	The discretization of~\eqref{vi_cont} then reads: Find  $u \in  K_h \coloneq \{  v_h\in  V_h\colon \left. v_h \right|_{\Gamma_S} \leq g_h \}$,  such that
	\begin{align}\label{vi_discrete}
		a(u_h, v_h-u_h) \geq f( v_h -u_h), \quad v_h \in K_h.
	\end{align}
For simplicity let us assume that $g$ is affine and $g_h = g$. In view of the homogeneous Dirichlet condition on $u$ this results in $g > 0$.

\subsection{Main results} \label{subsec:main_results}
$H^1(\Omega)$ error 
estimates of order $h$ for contact problems are given in~\cite{wohlmuth:05} under some regularity assumption on the active set, as well as more recently in~\cite{hild:12, hild:1X} for the 2D case under weaker assumptions on the solution.

For the case $\Omega \subset \R^3$, we have to assume some  regularity for the active set $\Gammaact$, in order to exclude a fractal active set. The assumption is similar to~\cite[Assumption 4.4]{wohlmuth:11}.  Given $\Sigma_h := \{ x\in \Gamma_S: \dist(x, \partial \Gammaact) \leq 2h\}$, we assume that 
\begin{equation}\label{eq:fractal_assumption}
\| u - g \|_{L^2(\Sigma_h)} 
\leq c h^{2-\varepsilon} \left| u  \right|_{H^{2-\varepsilon}(\Gamma_S)}, \quad u\in H^{2-\varepsilon}(\Gamma_S).
\end{equation} 
The abstract condition~\eqref{eq:fractal_assumption} is implied by the following criterion based on $\Gammaact$, see~\cite[Lemma 2.1]{wohlmuth:09}.
The active set  fulfils a cone property, has a piecewise $C^1$ boundary and there exists a $\delta_0 >0$ such that for all $0<\delta<\delta_0$ and $x\in \partial \Gammaact$ it holds $x + \delta n \not \in \Gammaact$, where $n$ is the outer unit normal of $\partial \Gammaact$ in $\Gamma_S$. See Figure~\ref{fig:regularity_gamaaact} for an illustration of the regularity condition. Note that for $\Omega \subset \R^2$, no similar assumption is necessary, due to recently introduced techniques in~\cite{hild:1X}.
 
 \begin{figure}[htbp]
\begin{center}
 \includegraphics[height=9em]{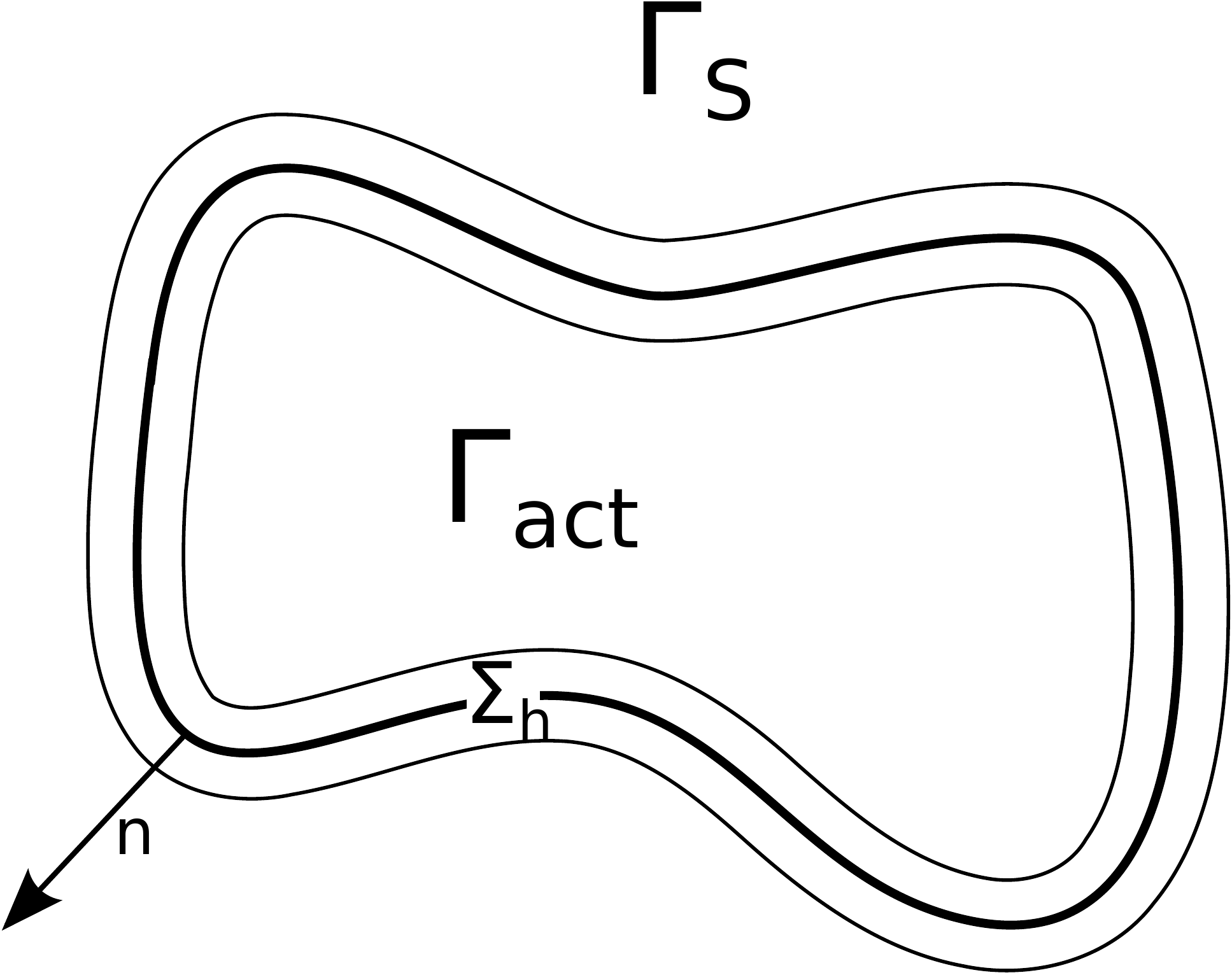}\hspace{4em}
 \includegraphics[height=8em]{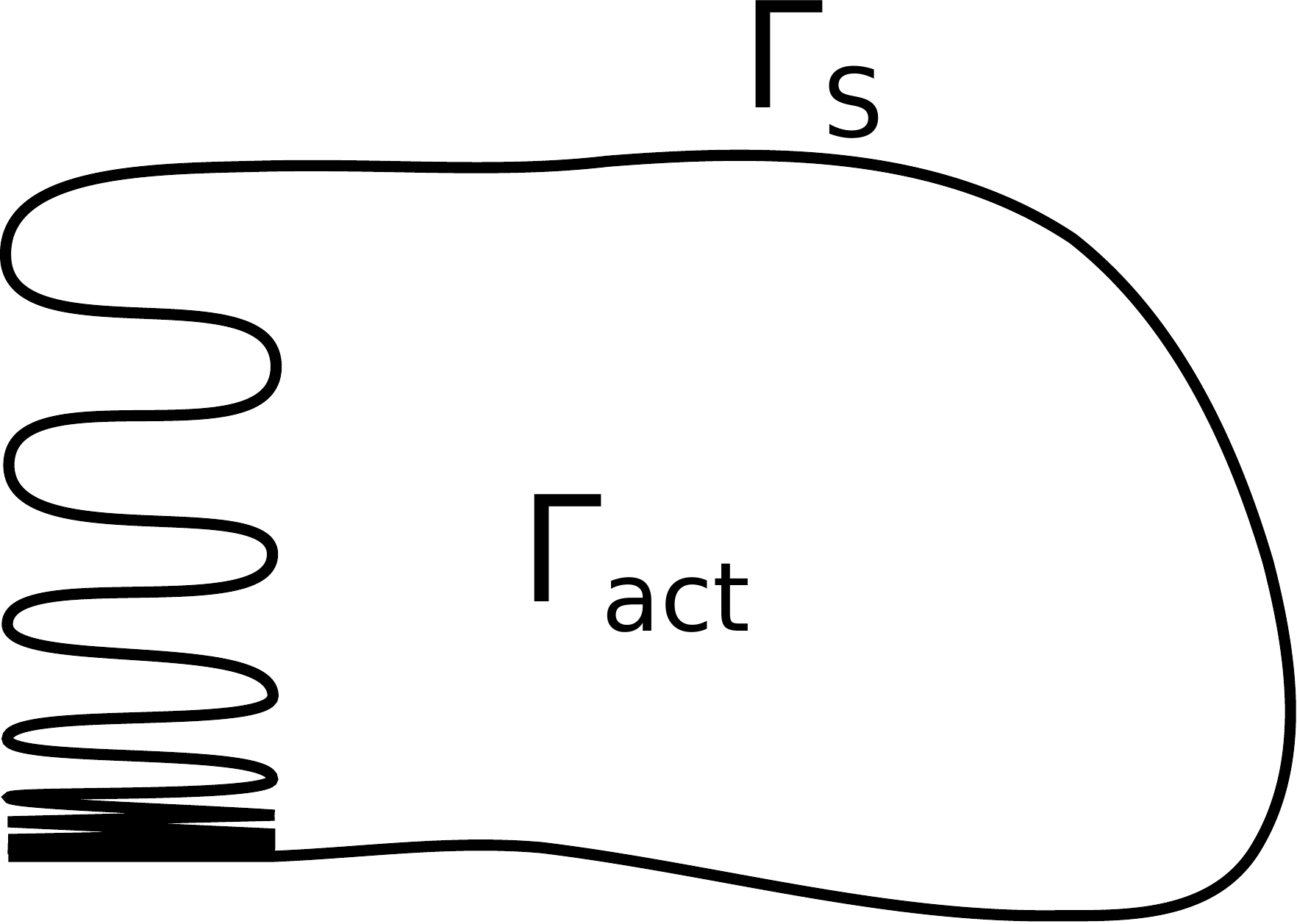}
\end{center}
 \caption{Illustration of the regularity assumption. Left: fulfilled condition and an illustration of the set $\Sigma_h$. Right: violated criterion on $\Gammaact$.}
 \label{fig:regularity_gamaaact}
 \end{figure}

In the following, $\varepsilon \in (0, 1/2]$ is fixed. 
Generic constants $0 < c, C < \infty$ are independent of the mesh size, but possibly dependent on the mesh regularity and $\varepsilon$.

The  main result of this paper is summarized in the following theorem and proved in the following sections.
\begin{theorem}\label{thm:main_result}
Let $u$ be the solution of~\eqref{vi_cont} and $u_h$ be the solution of~\eqref{vi_discrete}. Assuming $u\in H^{5/2-\varepsilon}(\Omega)$ and that Assumption~\eqref{eq:fractal_assumption} holds for $d=3$, then we get
\begin{align*}
		\| u - u_h \|_{H^{1/2}_{00}(\Gamma_S)} &\leq c  h^{3/2 - \varepsilon} \|u\|_{H^{5/2-\varepsilon}(\Omega)}.
	\end{align*} %Volle Norm mindestens wegen anisotropen Abschaetzung 
\end{theorem}
Based on this trace estimate, we can easily improve the $L^2(\Omega)$ estimate, up to the order $h^{3/2-\varepsilon}$. 
Additionally we show optimal approximation results for the boundary flux $\left. \partial_n u \right|_{\Gamma_S}$ in the natural $H^{-1/2}(\Gamma_S)$ norm.

%% file: 03_reformulations.tex
A crucial role in our analysis play three different but equivalent variational formulations.
Since  $g = g_h$ is affine, the pointwise condition $u_h \leq g_h$ can be reformulated in a variationally consistent way, using a biorthogonal dual basis. This choice yields the second variational formulation, a saddle point formulation, where the primal solution as well as the flux on the Signorini boundary are unknowns.
The third formulation, a variational formulation for the Schur complement posed on $\Gamma_S$, is adequate to bound the primal trace error.
However, the Schur complement of the discrete formulation differs from the Galerkin discretization of the continuous Schur complement.

\subsection{Saddle point formulation}
The second formulation, a saddle point problem, is  widely used for Signorini-type as well as contact problems. It can be obtained from the theory of constrained optimization, see for example~\cite{ekeland:99, kunisch:08}. 
Associated to the dual space of $W$, $M \coloneq H^{-1/2}(\Gamma_S)$, is the convex cone $M^+ \coloneq \{ \mu \in M\colon \langle v, \mu \rangle_{\Gamma_S} \geq 0, v \in W, v \geq 0\}$, where $\langle \cdot, \cdot \rangle_{\Gamma_S}$ denotes the duality pairing between $H^{1/2}_{00}(\Gamma_S)$ and $H^{-1/2}(\Gamma_S)$.

The saddle point problem reads: Find $(u, \lambda) \in  V \times  M^+$, such that
	\begin{subequations}\label{spp_cont}
	\begin{align}
		 a(u, v) + \langle v, \lambda\rangle_{\Gamma_S} &= f( v),  \quad v\in V, \label{spp_cont_line_1} \\
		 \langle u ,\mu - \lambda\rangle_{\Gamma_S} &\leq \langle  g , \mu - \lambda\rangle_{\Gamma_S}, \quad \mu \in M^+.
	\end{align}
	\end{subequations}

Let the vertices be enumerated such that the first $N_{M_h}$ vertices lie on $\Gamma_S$.
Associated to the primal Lagrange basis functions $\varphi_i, i = 1, \ldots, N_{M_h}$, which are supported on $\Gamma_S$, are biorthogonal basis functions $\psi_i \in L^2(\Gamma_S)$, $i = 1, \ldots, N_{M_h}$, satisfying $\langle \varphi_j, \psi_j \rangle_{\Gamma_S} = \delta_{ij}\,\langle \varphi_j, 1\rangle_{\Gamma_S} $, see for example~\cite{wohlmuth:01}. The discrete dual space $M_h$ is spanned by the biorthogonal basis functions $\psi_i \in L^2(\Gamma_S)$, and a uniform inf-sup stability for the discrete spaces $V_h$ and $M_h$ holds, see~\cite{wohlmuth:00}. 
  The convex cone $M^+$ is discretized as the positive span of the biorthogonal basis functions, i.e., $M_h^+ \coloneq \{\sum_{i=1}^{N_{M_h}}\alpha_i \psi_i, \alpha_i \geq 0\}$. 
 We note that a crosspoint modification is in practice not required due to our assumption that $\Gammaact$ is a compact subset of $\Gamma_S$.

The discretized saddle point formulation of~\eqref{spp_cont} then reads: Find
$(u_h, \lambda_h) \in  V_h \times  M_h^+$, such that
	\begin{subequations}\label{spp_discrete}
	\begin{align}
		 a(u_h, v_h) + \langle v_h, \lambda_h\rangle_{\Gamma_S} &= f( v_h),  \quad v_h\in V_h, \label{spp_discrete_line_1}  \\
		 \langle u_h ,\mu_h - \lambda_h\rangle_{\Gamma_S} &\leq \langle  g, \mu_h - \lambda_h\rangle_{\Gamma_S}, \quad \mu_h \in M_h^+.
	\end{align}
	\end{subequations}
	We point out, that $M_h\subset M$  but the discrete cone $M^+_h$ is not included in $M^+$.

\subsection{Reformulation as a Schur complement system} \label{subsec:reformulation}
Due to the fact, that the inequality constraint is solely located on the boundary, we can rewrite~\eqref{spp_cont} and~\eqref{spp_discrete} as Schur complement systems. 
On  the continuous level, we define the  Steklov--Poincar\'{e} operator by solving the Dirichlet problem 
\[
 -\Delta w_z = 0 \text{ in } \Omega, \quad w_z = 0 \text{ on } \Gamma_D, \quad w_z = z \text{ on } \Gamma_S,
 \]
for any $z\in H^{1/2}_{00}(\Gamma_S)$ and defining $S z \coloneq  \left. \partial_n w_z \right|_{\Gamma_S}$. The continuous Newton potential $\Nf = - \partial_n \widehat w_f$ is defined based on the solution of the homogeneous  Dirichlet problem $ -\Delta \widehat w_f = f$ in $\Omega$ and $\widehat w_f = 0$ on $\partial \Omega$. 
Based on these operators, we can formulate the Schur complement system which is a variational inequality on the Signorini boundary.
The primal trace $u_S \coloneq \left. u \right|_{\Gamma_S}\in K_{S} \coloneq \{v\in H^{1/2}_{00}(\Gamma_S)\colon  v \leq  g \}$ solves
\begin{equation}\label{schur_cont}
\langle v- u_S, S u_S  \rangle_{\Gamma_S} \geq \langle v- u_S, \Nf \rangle_{\Gamma_S},  \quad v \in  K_S.
\end{equation}

An equivalent characterization of the Steklov--Poincar\'e operator is possible as the Lagrange multiplier $\lambda_z = - Sz$ of a saddle point problem where $(w_z, \lambda_z) \in V\times M$ solves
\begin{subequations}\label{spp_SP_cont}
\begin{align}
a(w_z, v) + \langle v, \lambda_z  \rangle_{\Gamma_S} &= 0, \quad v\in V, \\
\langle w_z, \mu \rangle_{\Gamma_S} &= \langle z, \mu \rangle_{\Gamma_S}, \quad \mu \in M,
\end{align}
\end{subequations}
which corresponds to weakly imposed Dirichlet conditions, see~\cite{babuska:76}.
The continuous Newton potential can also be defined as the Lagrange multiplier of an analogue saddle point formulation.
The Steklov--Poincar\'{e} operator and the Newton potential map Dirichlet data and volume data to Neumann data, respectively. They have several applications, for example in domain decomposition and boundary element methods, see~\cite{quarteroni:99, toselli:05, steinbach:08}.

By using a mixed finite element approximation to the above Dirichlet problem~\eqref{spp_SP_cont}, we can define a mesh-dependent Steklov--Poincar\'e operator $S_h:W\rightarrow M_h$ by $S_hz := -\lambda_{z,h}$, where $(w_{z,h}, \lambda_{z,h}) \in V_h\times M_h$ solves
\begin{subequations}\label{spp_SP_discrete}
\begin{align}
a(w_{z,h}, v_h) + \langle v_h, \lambda_{z,h} \rangle_{\Gamma_S} &= 0, \quad v_h\in V_h, \\
\langle w_{z,h}, \mu_h \rangle_{\Gamma_S} &= \langle z, \mu_h \rangle_{\Gamma_S}, \quad \mu_h \in M_h.
\end{align}
\end{subequations}
An analogue discretization yields a mesh-dependent Newton potential $\Nfh$. Denote by $W_h$ the trace space of $V_h$. Up to scaling factors, the matrix formulation for $\left. S_h\right|_{W_h}$ and $\Nfh$ coincide with the discrete Schur complement system of the matrix formulation of~\eqref{vi_discrete} by construction.  
The uniform continuity of $S_h$ directly follows from the saddle point theory using the inf-sup stability of the discrete spaces, while the uniform $W_h$-ellipticity follows using  basic properties of discrete harmonic functions, e.g.,~\cite[Lemma 4.10]{toselli:05}. Precisely, it holds $\langle v_h, S_h  v_h \rangle_{\Gamma_S}  = a( w_{v,h} ,  w_{v,h})$, where $w_{v,h} \in V_h$ is the discrete harmonic extension of $v_h\in W_h$, hence $\langle v_h, S_h  v_h \rangle_{\Gamma_S}  = \left| w_{v,h}\right|_{H^1(\Omega)}^2 \geq c \|v_h\|_{H^{1/2}_{00}(\Gamma_S)}^2$.

The Schur complement system of~\eqref{vi_discrete} can be represented as an approximative discretization of~\eqref{schur_cont}. For $K_{S, h} := \{v_h \in W_h:  v_h \leq  g\}$, find $u_{S,h} \in K_{S, h}$, such that
\begin{equation}\label{schur_discrete_op}
\langle v_h- u_{S,h}, S_h u_{S,h}  \rangle_{\Gamma_S} \geq \langle v_h- u_{S,h}, \Nfh \rangle_{\Gamma_S},  \quad v_h \in  K_{S,h}.
\end{equation}
The three weak formulations~\eqref{vi_cont},~\eqref{spp_cont} and~\eqref{schur_cont} are equivalent as well as the three discrete variational problems~\eqref{vi_discrete},~\eqref{spp_discrete} and~\eqref{schur_discrete_op}.

%% file: 04_strang.tex
While $u$ solves the variational inequality~\eqref{schur_cont} with the operators $S$ and $N$, the discrete solution $ u_h $ solves the variational inequality~\eqref{schur_discrete_op} with the mesh dependent operators $S_h$ and $N_h$. 
In this subsection, we show that the $H^{1/2}_{00}(\Gamma_S)$ error can be bounded by two terms. 
The first term is the $H^{-1/2}(\Gamma_S)$ norm of the difference between $\Nf - S\! \left( \left. \Tr u \right|_{\Gamma_S} \right)  = \lambda$ and $\Nfh - S_h\! \left( \left. \Tr u \right|_{\Gamma_S} \right) \eqcolon \widetilde \lambda_h \in M_h$. Note that $\widetilde \lambda_h$ is the discrete dual solution of the linear saddle point problems defining the Dirichlet--Neumann map, see~\eqref{spp_SP_discrete}. Associated with $\widetilde \lambda_h$ is $\widetilde u_h = {\widehat{w}_{f,h}} + w_{\left. \Tr u \right|_{\Gamma_S}, h}$ and $(\widetilde u_h, \widetilde \lambda_h) \in V_h \times M_h$ solves
\begin{align*}
	a( \widetilde u_h , v_h) + \langle v_h,\widetilde \lambda_h  \rangle_{\Gamma_S} &= f(v_h), \quad v_h \in V_h, \\
	\langle \widetilde u_h, \mu_h \rangle_{\Gamma_S} &= \langle u,\mu_h \rangle_{\Gamma_S}, \quad \mu_h \in M_h.
\end{align*}

The second term is the discretization error of the  variational inequality on the boundary~\eqref{schur_cont}. Let $ \bar u_h \in K_{S,h}$ be such that
\begin{align}\label{eq:unperturbed}
\langle  v_h- \bar u_h, S  \bar u_h  \rangle_{\Gamma_S} \geq \langle  v_h- \bar u_h, \Nf \rangle_{\Gamma_S},  \quad  v \in  K_{S, h}.
\end{align}
\begin{lemma} \label{lem:strang}
The trace error of the Signorini problem~\eqref{vi_cont} can be bounded by
\begin{align*}
\| u - u_h \|_{H^{1/2}_{00}(\Gamma_S)} \leq   c \| \lambda - \widetilde \lambda_h \|_{H^{-1/2}(\Gamma_S)} + c \| u - \bar u_h\|_{H^{1/2}_{00}(\Gamma_S)}.
\end{align*}
\end{lemma}
\mybeginproof
The proof of this lemma follows the lines of~\cite[Theorem 3.2]{steinbach:12b}. Since the proof is fundamental, we work  it out. 
We start with the trivial triangle inequality \[\|u-u_h\|_{H^{1/2}_{00}(\Gamma_S)} \leq \|u - \bar u_h\|_{H^{1/2}_{00}(\Gamma_S)} + \|\bar u_h - u_h\|_{H^{1/2}_{00}(\Gamma_S)}.\] For the second term $\bar u_h - u_h$, we use the $W_h$-ellipticity of the mesh-dependent Steklov--Poincar\'{e} operator and apply the  variational inequalities~\eqref{schur_discrete_op} and~\eqref{eq:unperturbed}:
\begin{align*}
&c \|\bar u_h - u_h\|_{H^{1/2}_{00}(\Gamma_S)}^2 \leq \langle \bar u_h - u_h, S_h(\bar u_h - u_h)\rangle_{\Gamma_S}  \notag\\
&\qquad \leq   \langle \bar u_h - u_h, S_h\bar u_h\rangle_{\Gamma_S} + \langle \bar  u_h - u_h, \Nf - \Nfh \rangle_{\Gamma_S} - \langle \bar u_h - u_h, S  \bar u_h \rangle_{\Gamma_S}\notag\\
&\qquad = \langle  \bar u_h - u_h,  \Nf - S  \bar u_h - (\Nfh - S_h\bar u_h) \rangle_{\Gamma_S} \notag \\&\qquad \leq
\|  \bar u_h - u_h\|_{H^{1/2}_{00}(\Gamma_S)} \|  \Nf - S  \bar u_h - (\Nfh - S_h\bar u_h) \|_{H^{-1/2}(\Gamma_S)} .
\end{align*}
Using the boundedness of the operators and once again the triangle inequality, we get
\begin{align*}
\|\bar u_h - u_h\|_{H^{1/2}_{00}(\Gamma_S)} &\leq  
 c \|  \Nf - S  u - (\Nfh - S_h u) \|_{H^{-1/2}(\Gamma_S)} \notag\\&\qquad + c \| S(u -  \bar u_h) \|_{H^{-1/2}(\Gamma_S)}  +  c\| S_h(u -  \bar u_h)\|_{H^{-1/2}(\Gamma_S)} \notag\\ 
&\leq c \| \lambda - \tilde \lambda_h\|_{H^{-1/2}(\Gamma_S)} + c \|u - \bar u_h \|_{H^{1/2}_{00}(\Gamma_S)}. \notag \myendproof %\hfill \square \qquad \endproof \notag
\end{align*}
A bound of $u - \bar u_h$ can be shown using Falk's lemma, see,~\cite[Theorem 1]{falk:74}, which is an analogue result to C\'ea's lemma for variational inequalities.  
Since the discretization of the variational inequality is conforming in the sense, that $K_{S, h} \subset K_S$, Falk's lemma reads 
\begin{align}\label{falks_lemma}
\| u - \bar u_h \|_{H^{1/2}_{00}(\Gamma_S) }& \leq  
 c \inf_{v_h \in K_{S,h}}\left( \| u - v_h \|_{H^{1/2}_{00}(\Gamma_S)} + \langle \lambda , u - v_h \rangle_{\Gamma_S}^{1/2} \right).
\end{align}

\begin{lemma} \label{lem:estimate_unperturbed}
Let $u \in K_S$ be the solution to the variational inequality on the boundary~\eqref{schur_cont} and $\bar u_h \in K_{S, h}$ the Galerkin approximation, see Equation~\eqref{eq:unperturbed}. Assuming $u\in  H^{5/2-\varepsilon}(\Omega)$ and that Assumption~\eqref{eq:fractal_assumption} holds for $d=3$, then we get
\[
\| u - \bar u_h\|_{H^{1/2}_{00}(\Gamma_S)} \leq c h^{3/2- \varepsilon}   \left|u\right|_{H^{5/2-\varepsilon}(\Omega)}.\]
\end{lemma}
\mybeginproof
This type of estimate was already considered  in the context of boundary element methods, in~\cite[Theorem 3.1]{spann:93} and~\cite[Section 3]{steinbach:12}, where additional assumptions on the boundary of the active set were made. To keep this article self-contained, we present a proof, based on techniques for $H^1(\Omega)$ estimates.

In this proof only the Signorini boundary $\Gamma_S$ is considered, so any notation refers to $\R^{d-1}$. We introduce the triangulation $\mathcal{T}_h^S$ on the Signorini boundary which is induced by the triangulation of $\Omega$.
We note, that the induced triangulation on the Signorini boundary is also shape-regular and denote by $h_T$ the diameter of an element $T\in \mathcal{T}_h^S$.

The proof is carried out for the case $d=2$ and $d=3$ separately. We start with $d=2$ and use recently shown local $L^1$  and $L^2$ estimates from~\cite{hild:1X}.
Using $v_h = I_h  u \in K_{S, h}$, the piecewise linear nodal interpolation, in Falk's lemma~\eqref{falks_lemma}, it remains to bound 
\begin{align*}
\langle \lambda , u - I_h u \rangle_{\Gamma_S} = \sum_{T \in \mathcal{T}_h^S} \int_{T} \lambda( u - I_hu ) ~ \mathrm{d}\Gamma.
\end{align*}
One of the main ideas of this proof is to derive two estimates for each element, where dependent of the measure of the active area $|T \cap \Gammaact|$ one of the two estimates is applied. 
Given any $T  \in \mathcal{T}_h^S$, define the local active area $T^{\mathrm{act}} = T\cap \Gammaact$ and  the local inactive area $T^{\mathrm{inact}} = T \backslash T^{\mathrm{act}}$, see Figure~\ref{fig:element_splitting}.
\begin{figure}[!htbp]
\begin{center}
\includegraphics[width=.7\textwidth]{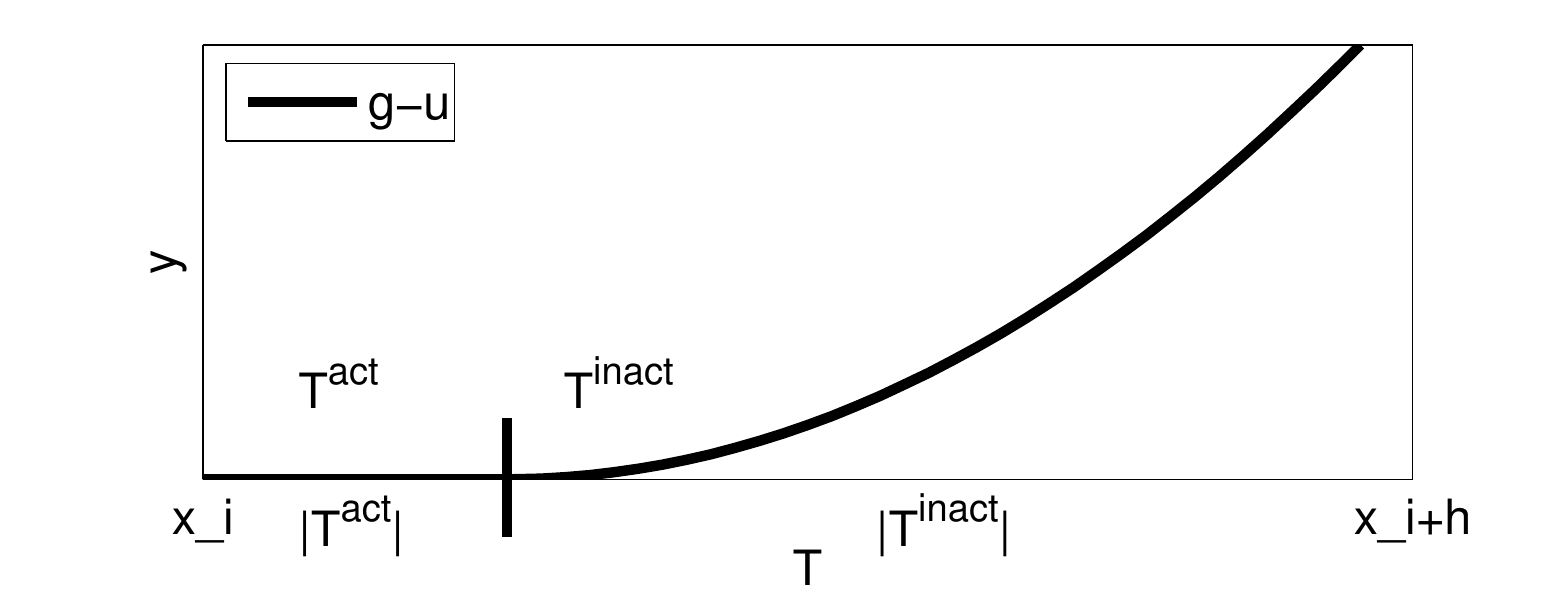}
\end{center}	
\caption{Active and inactive area within one element $T$.}
\label{fig:element_splitting}
\end{figure}
Note that by construction only the elements with $\left| T^{\mathrm{act}} \right| > 0$ and $\left| T^{\mathrm{inact}} \right| > 0$ are of interest. 
Recently developed nonstandard estimates for $u$ and $\lambda$, see,~\cite[Lemma 1, Lemma 2]{hild:1X}, yield
\begin{align*}
 &\int_{T} \lambda( u - I_hu ) ~ \mathrm{d}\Gamma  \\&\qquad \leq c \min\left( \frac{\left|T\right|^{1/2}}{ \left| T^{\mathrm{inact}} \right|^{1/2}}, \frac{\left|T\right|^{1/2}}{ \left| T^{\mathrm{act}} \right|^{1/2}} \right) h_T^{3-2 \varepsilon} \left( \left| \lambda \right|_{H^{1-\varepsilon}(T)}^2 +  \left| u \right|_{H^{2-\varepsilon}(T)}^2 \right).
\end{align*}

Since $\left| T^{\mathrm{act}} \right| + \left| T^{\mathrm{inact}} \right| = h_T$, one of the measures is greater equal than $h_T/2$.
Summing over the elements and applying the trace inequality yields the desired estimate.

For the second case, $d=3$, we define a modified interpolation operator, as used in~\cite{wohlmuth:05}. Note that we are interested in a higher approximation order, so we cannot directly apply these results. Let $x_i$, $i=1,\ldots,N_{M_h}$, denote the vertices in the interior of $\Gamma_S$, we define $\tilde{I}_h u \in K_{S,h}$ as
\[
(\tilde{I}_h u)(x_i) =
\begin{cases}
u(x_i), \quad \text{ for } \supp \varphi_i \subset \overline{ \Gamma_C \backslash \Gammaact}, \\
g(x_i), \quad \text{ otherwise,}
\end{cases}
\]
see Figure~\ref{fig:mod_interpolation}.
\begin{figure}[!htbp]
\begin{center}
\includegraphics[width=.7\textwidth]{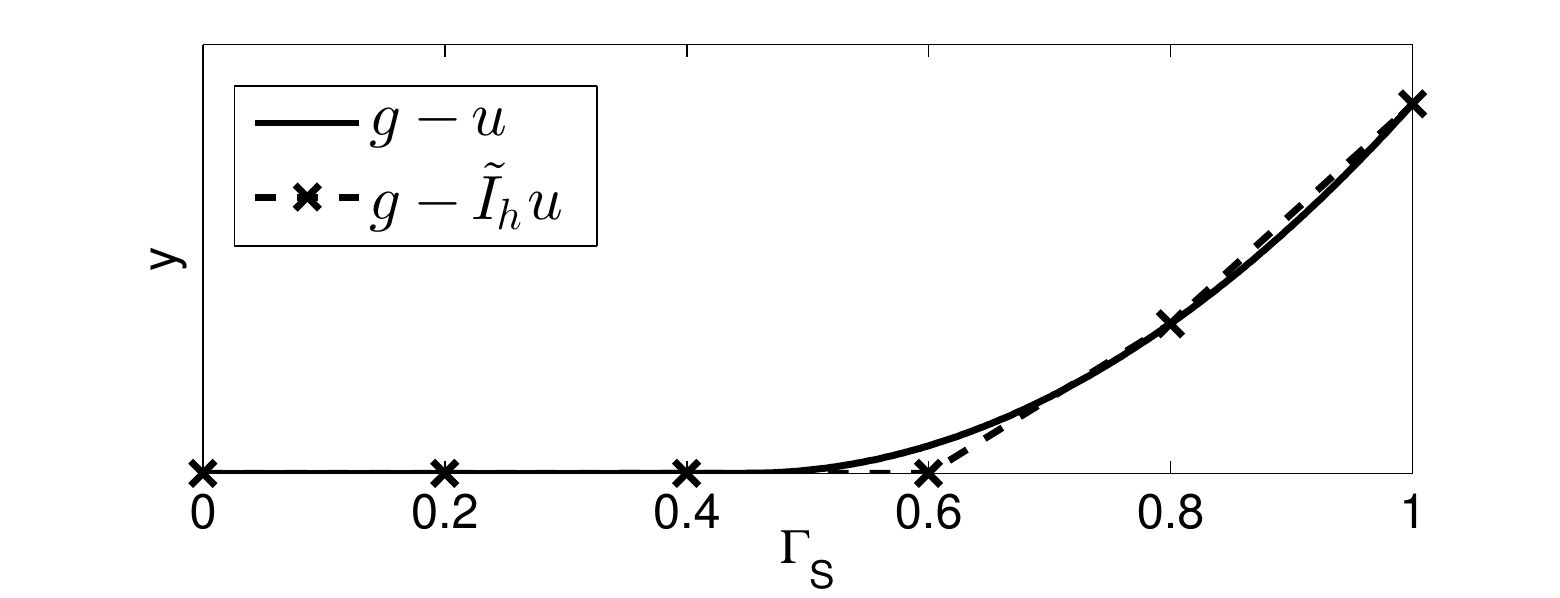}
\end{center}	
\caption{Sketch of the modified interpolation operator $\tilde{I}_h$}
\label{fig:mod_interpolation}
\end{figure}
Note that, for $h$ sufficiently small, $\dist(\Gammaact, \partial \Gamma_S) > h$ and the operator is well defined. Again, we apply Falk's lemma, see~\eqref{falks_lemma}.
By construction, $\tilde{I}_h u$ fulfils $\langle  u - \tilde{I}_h u , \lambda \rangle_{\Gamma_S} \leq 0$ and it remains to estimate $\|u - \tilde{I}_h u \|_{H^{1/2}_{00}(\Gamma_S)}$.  Using the piecewise linear nodal interpolation $I_h u$ 
\[
\|u - \tilde{I}_h u \|_{H^{1/2}_{00}(\Gamma_S)} \leq
\|u - I_h u \|_{H^{1/2}_{00}(\Gamma_S)} + \|I_h u - \tilde{I}_h u \|_{H^{1/2}_{00}(\Gamma_S)},
\]
only the second term remains to be considered. An inverse estimate yields
\[
\|I_h u - \tilde{I}_h u \|_{H^{1/2}_{00}(\Gamma_S)}^2 \leq c h \sum_{k=1}^{N_{M_h}} (u(x_k) -  (\tilde{I}_h u)(x_k) )^2 , 
\]
where 
$u(x_k) =  (\tilde{I}_h u)(x_k) $ if $\dist(x_k, \partial \Gammaact) > h$. Finally, using $u(x) - g(x) = 0$ for $x\in \Gammaact$ and  applying the assumption~\eqref{eq:fractal_assumption}, we get
\begin{align*}
\|I_h u - \tilde{I}_h u \|_{H^{1/2}_{00}(\Gamma_S)}^2 &\leq c h \sum_{k=1}^{N_{M_h}} (u(x_k) - g(x_k) )^2 \leq c h^{-1} \|u - g \|_{L^2(\Sigma_h)}^2 \\& %\leq c h^{3-2\varepsilon} \left| \nabla u \right|_{L^2(\Sigma_h)}^2 
\leq c h^{3-2\varepsilon} \left| u\right|_{H^{2-\varepsilon}(\Omega)}^2. \myendproof
\end{align*}

%% file: 05_boundary_estimates.tex
In   this section, an upper bound for $\| \lambda - \widetilde \lambda_h \|_{H^{-1/2}(\Gamma_S)}$ is shown which concludes the primal trace estimate in Lemma~\ref{thm:main_result}. The Lagrange multiplier arises from a linear Dirichlet problem with a weak enforcement of the boundary values which is covered by the problem formulation in~\cite{wohlmuth:12}. However, the required regularity of $B^{5/2}_{2,1}(\Omega)$ is not given in our case. Thus we have to generalize these results. We follow the lines of~\cite{wohlmuth:12} but will not work with the Besov space $B^{5/2}_{2,1}(\Omega)$. Reducing the regularity from $B^{5/2}_{2,1}(\Omega)$ to $H^{5/2-\varepsilon}(\Omega)$ automatically results in a reduced convergence order, but we do not loose a log-term. 

The first two subsections collect some technical tools for the  proof which is carried out in Subsection~\ref{sec:main_proof}. Firstly, for a Scott--Zhang operator, we show optimal approximation results in anisotropic norms. Secondly, for two dual problems, estimates in these norms are shown. As a corollary of the main result, we show improved rates in the $L^2(\Omega)$ norm. 

\subsection{Anisotropic norms and quasi-interpolation results} \label{subsec:anisotropic}
Estimating the dual solution on the boundary can be related to bounds of the primal solution in a neighbourhood of $\Gamma$. We define strips around the boundary of width $\delta$ by $\strip(\delta) \coloneq \{x\in \Omega: \dist(x, \Gamma) \leq \delta\}$. 
Using a dual Neumann problem with local volume data,  we can relate the dual error to the primal error on a strip $\strip(ch)$. 
As a technical tool to derive local error estimates for the dual problems on these strips, we use anisotropic norms as in~\cite{ wohlmuth:12, human:13,waluga:13}.  We simplify the original definition, which was based on a technical decomposition of the domain into "cylinders". Instead, we use an intuitive decomposition into triangles and pyramids, based on the faces of the polygonal domain. 
 
For a formal definition, we first decompose the domain $\Omega$ into a set of patches which are triangles if $d=2$ and pyramids if $d=3$. 
Each patch is supposed to connect one facet with the barycentre of $\Omega$.
Since $\Omega$ is convex the barycentre $x_c$ lies in the interior of $\Omega$. 
Let an enumeration of the facets be given by $\gamma_i$, $i = 1,\ldots,N_\gamma$ and consider one facet $\gamma_i$. The patch $\Omega_i$ is the triangle respectively pyramid with $\gamma_i$ as base side and $x_c$ as the top. Obviously $\overline{\Omega} = \cup_{i=1}^{N_\gamma} \overline{\Omega_i}$, see Figure~\ref{fig:patches}.
\begin{figure}[htbp]
\begin{center}
\includegraphics[width=.6\textwidth]{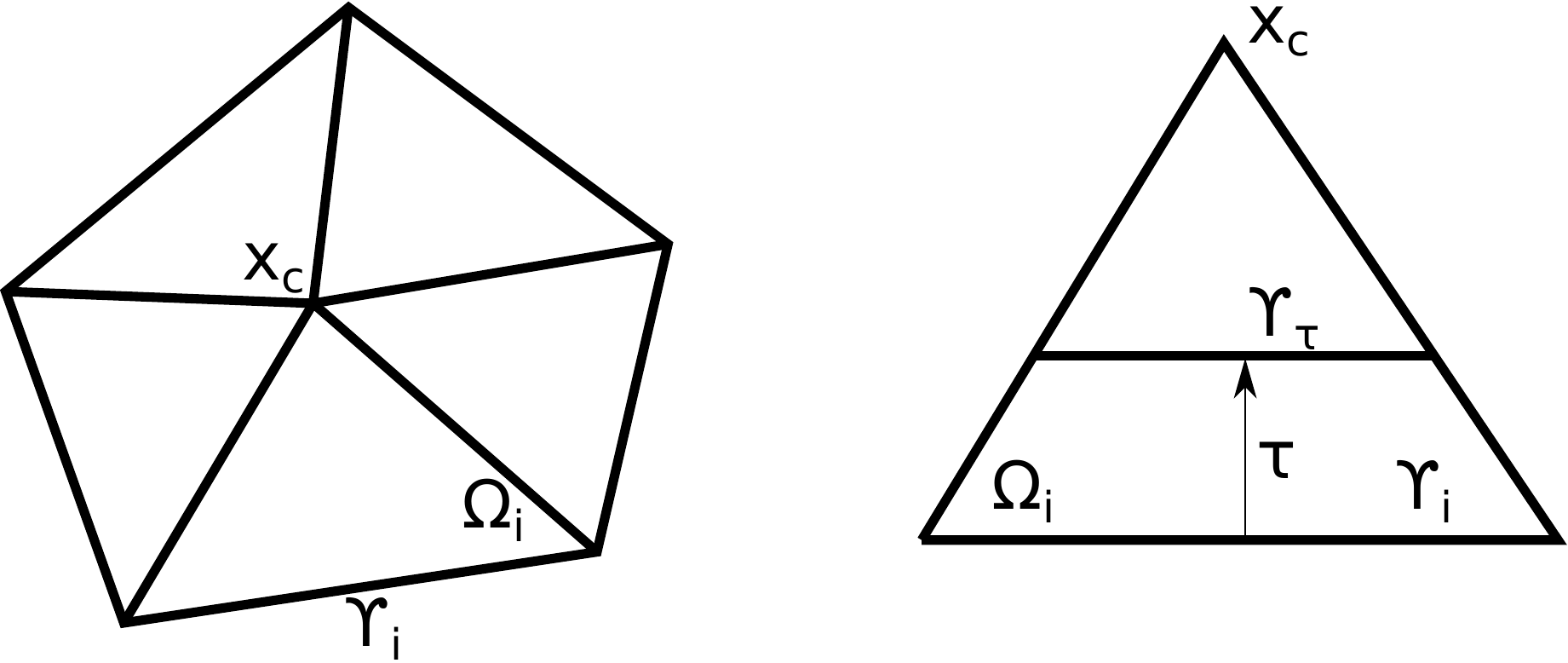}
\end{center}
\caption{Left: Decomposition of a 2D domain into the patches. Right: One patch after a suitable rotation and the necessary notation}
\label{fig:patches}
\end{figure}
For each patch $\Omega_i$, we define the anisotropic norm $L(p,2;\Omega_i)$ based on a decomposition of the patch into a $(d-1)$-dimensional part parallel to the facet $\gamma_i$ and the one dimensional distance to the facet. Given $i \in \{1,\ldots,N_\gamma\}$, without any loss of generality, we assume that $\gamma_i$ lies in the $x_1,\ldots, x_{d-1}$-plane and $\Omega$ lies in the positive half space $\{(x', \tau), x' \in \R^{d-1}, \tau \geq 0\}$. 
We denote  $\gamma_{ \tau} \coloneq \{  (x',  \tau) \in \Omega_i, x'\in \mathbb{R}^{d-1}\}$ for $ \tau \geq 0$, the part parallel to $\gamma_i$.
 We have $\gamma_{ \tau} = \emptyset$ for  $ \tau < 0$ and $ \tau \geq D$, where $D$ is the diameter of $\Omega$.
By the Fubini--Tonelli formula, the integral over $\Omega$ can be decomposed as
\[
\int_{\Omega_i} v~\mathrm{d}x = \int_{\tau=0}^D \int_{\gamma_{\tau}} v~\mathrm{d}\mu ~\mathrm{d}\tau,
\]
where $\mathrm{d}\mu$ denotes the $(d-1)$-dimensional Lebesgue measure.
We define  anisotropic norms $L(p,2; \Omega_i)$, $1\leq p \leq \infty$, 
by
\begin{align*}
\lVert v \rVert_{L(p,2; \Omega_i)}^p &\coloneq \int_{\tau = 0}^D \left( \int_{\gamma_{\tau} } v^2 \mathrm{d}\mu \right)^{p/2} ~\mathrm{d}\tau,\quad 1\leq p < \infty, \\
\lVert v \rVert_{L(\infty,2; \Omega_i)} &\coloneq \sup_{\tau \in(0, D)} \left( \int_{\gamma_{\tau} } v^2 \mathrm{d}\mu \right)^{1/2}.
\end{align*}
Adding the components of each patch, we define anisotropic norms on $\Omega$:
\begin{align*}
\lVert v \rVert_{L(p,2)}^p &\coloneq \sum_{i=1}^{N_\gamma} \lVert v \rVert_{L(p,2; \Omega_i)}^p ,\quad 1\leq p < \infty, \\
\lVert v \rVert_{L(\infty,2)} &\coloneq \max_{i=1,\ldots, N_\gamma} \lVert v \rVert_{L(\infty,2; \Omega_i)}.
\end{align*}
 Note that the patches cover $\Omega$ without any overlap and the $L(2,2)$ norm coincides with the $L^2(\Omega)$ norm.
 
The H\"older inequality $\int_{\Omega} f g~\mathrm{d}x \leq \|f\|_{L(p,2)} \|g\|_{L(q,2)}$ for $1/p + 1/q = 1$ follows from the one-dimensional H\"older inequality. Furthermore an interpolation result analogue to $L^p$ spaces is valid.
\begin{lemma} \label{lem:interpolation_anisotropic}
For $1<p<\infty$
and $1/p + 1/{p'} = 1$, it holds
\[
L(p,2) = (L(1,2), L(\infty,2) )_{{1/p'}, p}.
\]
\end{lemma}
\begin{proof}
For convenience of the reader, we sketch the main steps.  
Consider any patch $\Omega_i$, $i \in \{1,\ldots,N_\gamma\}$.
For any $1\leq q \leq \infty$ and $v \in L(q,2;\Omega_i), I=(0,D)$, consider $f_v \in L^q(I)$ which is defined for almost every $\tau \in I$ by $f_v(\tau) \coloneq \| v \|_{L^2(\gamma_{\tau})}$. It holds $\| v\|_{L(q,2;\Omega_i)} = \| f_v \|_{L^q(I)}$, and we can show the equality of the two $K$-functionals
\begin{align*}
K(t,v;L(1,2;\Omega_i), L(\infty,2;\Omega_i) )&= \inf_{v = v_0 + v_1}( \|v_0 \|_{L(1,2;\Omega_i)} + t\| v_1 \|_{L(\infty,2;\Omega_i)} ),\\
K(t, f_v; L^1(I), L^\infty(I)) &= \inf_{f_v = f_0 + f_1}( \|f_{0} \|_{L^1(I)}+ t\| f_{1} \|_{L^\infty(I)} ),
\end{align*}
and use the standard $L^p$-interpolation $L^p(I) = (L^1(I), L^\infty(I))_{1/p', p}$.

On the one hand, any decomposition $f_v = f_0 + f_1$  directly implies a decomposition by $v_i(x', \tau) \coloneq v(x',\tau) f_i(\tau) / f_v(\tau)$ for $x'\in \R^{d-1}$. The case $f_v(\tau) = 0$ is trivial and can be excluded. It holds $v = v_0 + v_1$ and $f_{v_i} = f_i$. As a consequence
\[
K(t,v;L(1,2;\Omega_i), L(\infty,2;\Omega_i) ) \leq K(t, f_v; L^1(I), L^\infty(I)).
\]
 On the other hand for any decomposition $v = v_0 + v_1$ it holds
\[
f_{v_0}(\tau) + f_{v_1}(\tau) =  \| v_0 \|_{L^2(\gamma_\tau)} +  \| v_1 \|_{L^2(\gamma_\tau)} \geq  \| v_0 + v_1 \|_{L^2(\gamma_\tau)}  = f_v(\tau).\]
Hence, the decomposition of $f_v$ by 
\[f_i(\tau) \coloneq f_{v_i}(\tau) ~ \frac{f_v(\tau)}{f_{v_1}(\tau) + f_{v_2}(\tau) } \leq f_{v_i}(\tau)\]
yields  $\|f_0 \|_{L^1(I)}\leq  \|v_0 \|_{L(1,2)}$ as well as $\| f_1 \|_{L^\infty(\Omega)} \leq \| v_1 \|_{L(\infty,2)}$. This implies 
\[
K(t,v;L(1,2;\Omega_i), L(\infty,2;\Omega_i) ) \geq K(t, f_v; L^1(I), L^\infty(I))
\]
and concludes the equality of both $K$-functionals.

Since the patches cover $\Omega$ without any overlap, the interpolation property for $L(p,2)$ follows. 
\end{proof}

As a preliminary to our analysis, we state approximation results of a Scott--Zhang type quasi-interpolation operator in the anisotropic norms.   
We consider $P_h\colon V \rightarrow V_h$ as in~\cite{scott:90}, based on the biorthogonal basis on $\Gamma_S$, preserving the homogeneous Dirichlet data on $\Gamma_D$. The boundary values are preserved such that $\left. P_h v \right|_{\Gamma_D} = 0$ and  $\langle P_h v , \mu_h \rangle_{\Gamma_S} = \langle v, \mu_h \rangle_{\Gamma_S}$ for $\mu_h \in M_h$.
On $\Gamma_S$, optimal order $L^2$ approximation properties 
\begin{equation}\label{sz:approximation_a_eq}
\| v - P_h v \|_{L^2(\Gamma_S)} \leq c h^{2-\varepsilon} \left| v \right|_{H^{2-\varepsilon}(\Gamma_S)}
\end{equation}
for $v \in V\cap H^{5/2-\varepsilon}(\Omega)$ are given.  
An approximation result in the $L(q, 2)$ norm is given by the following lemma.
\begin{lemma}
\label{sz:approximation_b}
For $v \in V\cap H^{5/2-\varepsilon}(\Omega)$, $ q = \varepsilon^{-1} \geq 2$, it holds
\begin{equation*}%\label{sz:approximation_b_eq}
\| \nabla\left( v -  P_h v \right) \|_{ L(q, 2) } \leq c h \| v\|_{ H^{5/2-\varepsilon}(\Omega) }.
\end{equation*}
\end{lemma}
\begin{proof}
Since the $L(2,2)$ norm coincides with the $L^2(\Omega)$ norm, we have the standard approximation result
\begin{equation*}
\| \nabla\left( v -  P_h v \right) \|_{ L(2, 2) } \leq c h \left| v\right|_{ H^2(\Omega) }.
\end{equation*}
For $q > 2$, we show the estimate by an interpolation argument, using the $L(2,2)$ and the $L(\infty, 2)$ estimate.
For the $L(\infty, 2)$ norm, we can easily adapt the proof in~\cite[Lemma 4.1]{wohlmuth:12} using local approximation results of the Scott--Zhang operator~\cite[Equation 4.3]{scott:90}. 
For any patch $\Omega_i$, $i\in \{1,\ldots, N_\gamma\}$ and $\tau > 0$, we first define two strips around $\gamma_\tau$. A strip of width $2\delta$ is defined by
$\strip_i(\delta,  \tau) \coloneq \{x\in \Omega: \dist(x, \gamma_{ \tau}) \leq \delta \}$
and a discrete neighbourhood can be constructed by the elements intersecting $\gamma_\tau$: $\mathcal{I}_\tau \coloneq \{T\in \mathcal{T}_h: \gamma_\tau \cap \overline{T}\neq \emptyset\}$. 
Note, that we cannot expect $\strip_i(\delta, \tau) \subset\Omega_i$, but this inclusion is not necessary for our analysis.
Using these strips, local estimates of the Scott--Zhang operator yield
\begin{align*}
\|\nabla( v - P_h v )\|_{L^2(\gamma_{\tau})}^2 &\leq c \sum_{T\in \mathcal{I}_\tau} \left( \frac1h \|\nabla( v - P_h v )\|_{L^2( T )}^2 + h  \| \nabla^2 ( v - P_hv ) \|_{L^2(T)}^2 \right) \notag \\
&\leq c h \left| v \right|^2_{H^2(\strip(\tilde c h,\tau))}
\leq c h^2 \|v\|^2_{ B_{2,1}^{5/2}(\Omega) },
\end{align*}
where in the last step~\cite[Lemma 2.1]{wohlmuth:09} was used. Consequently, we have
\begin{equation*}
\| \nabla\left( v -  P_h v \right) \|_{ L(\infty, 2) } \leq c h \| v\|_{ B_{2,1}^{5/2}(\Omega) }.
\end{equation*}

To show this estimate also for interpolation spaces, we apply the interpolation property~\cite[Lemma 22.3]{tartar:07}. 
By the reiteration theorem and Lemma~\ref{lem:interpolation_anisotropic}, we have the interpolation representations $L(q,2) = (L(2,2), L(\infty,2) )_{1-2\varepsilon, q}$ and the similar term $H^{\frac52 - \varepsilon}(\Omega)= (H^2(\Omega), B_{2,1}^{5/2}(\Omega) )_{1-2\varepsilon, 2} \subset (H^2(\Omega), B_{2,1}^{5/2}(\Omega) )_{1-2\varepsilon, q}$. As a consequence the stated estimate is also valid in the interpolated spaces.
\end{proof}% Mehr Tartar?

\subsection{Dual problems}
In this subsection, we follow the lines of~\cite[Section 5]{wohlmuth:12} and define a dual Dirichlet problem with locally supported data.
For $v\in L^2(\Omega)$, $\supp v \subset \overline{ \strip(h) }$, we denote by $T^Dv$ the solution operator of
\begin{equation} \label{dual_dirichlet}
	-\Delta w = v \quad \text{ in } \Omega, \qquad w = 0 \quad \text{ on } \Gamma,
\end{equation} 
i.e., $T^Dv = w$.

In contrast to~\cite{wohlmuth:12}, we cannot assume $B^{5/2}_{2,1}(\Omega)$ regularity for the solution of~\eqref{vi_cont}, but only $H^{5/2-\varepsilon}(\Omega)$ regularity. Naive interpolation of the final estimate does not yield optimal results but an additional log-term. 
For optimal results, we need the stronger estimate given in the following lemmas. In the next lemma, we state a regularity estimate in a weighted Sobolev space using the local support of the data of the dual problem. Based on this estimate, we then state an approximation result for the Galerkin approximation of the dual solution in an anisotropic norm.

\begin{lemma} \label{lem:dual_regularity}
For $v \in L^2(\Omega),~\supp v \subset \overline{ \strip(h) }$ and $w \coloneq T^Dv$ there exists $\tilde c$ independent of $v$ and $h$, such that
\[
\| \delta_\Gamma^{1/2 - \varepsilon/2} \nabla^2 w\|_{L^2(\Omega \backslash \strip(\tilde c h) )} \leq c h^{1/2-\varepsilon/2} \|v\|_{L^2(\Omega)},
\]
where $\delta_\Gamma$ is the distance function to $\Gamma$.
\end{lemma}
\begin{proof}We follow the idea of~\cite[Lemma 5.4]{wohlmuth:12}, but instead of several local translations of $w$, we consider a global scaling of the coordinate system. 
To exploit the local data of the dual problem, we choose a sufficiently large scale factor such that the transformation of $w$ is harmonic in a neighbourhood of $\Omega$.  This allows us to apply interior regularity results for the transformation of $w$, see \cite[Theorem 8.8]{gilbarg:01}:
\begin{equation} \label{eq:interior_regularity}
\| \nabla^2 z \|_{L^2(B_1)} \leq c \| z \|_{H^1(B_{1+\rho})},
\end{equation}
for $-\Delta z = 0$ on $B_{1+\rho}$, a ball of radius $1+\rho$ for a fixed $\rho>0$.

Without loss of generality, assume that the barycentre of $\Omega$ is the origin of the coordinate system. 
For  sufficiently small $h$, we define a neighbourhood of $\Omega$ by a scaling $\widetilde \Omega \coloneq \{ ( 1 + 4 C_1 h)x: x \in \Omega \}$. 
Since we estimate $w$ only on $\Omega \backslash \strip(\tilde c h)$, where $\tilde c$ is selected later, we can choose the scale factor appropriately. 
 The constant $C_1$ is sufficiently large, but fixed and independent of $h$, such that for $x \in \strip(h)$ it holds $  (1+2 C_1 h) x  \not \in \Omega$.
We scale $w$ to a function on this neighbourhood by $\widetilde w: \widetilde \Omega \rightarrow \R$, $\widetilde w( x ) \coloneq w(x / (1 + 4 C_1 h  ) ) $. 

Note that the introduced scaling preserves harmonic functions, more precisely for $x\in \Omega$ and $h < 1/(2C_1)$, we have $(1+C_1h)/(1+4C_1h) x \in \Omega\backslash \strip(h)$, and thus 
\begin{equation*} %\label{eq:scaled_regularity}
 \Delta \widetilde w = 0 \text{ at } (1+C_1 h) x, \quad x\in \Omega.
\end{equation*}
Since the scale factor is uniformly bounded, it also preserves Sobolev norms, i.e., 
\[c \|\widetilde w \|_{H^\sigma(\widetilde \Omega)} \leq  \| w \|_{H^\sigma(\Omega)}  \leq C  \|\widetilde w \|_{H^\sigma(\widetilde \Omega)} , \sigma \in \{0, 3/2\}.\]

To apply the transformation $\widetilde w$, we choose $\tilde c$ sufficiently large such that the transformation of $\Omega \backslash \strip(\tilde c h)$ is a subset of $\Omega\backslash \strip(h)$ and thus
\[
\| \delta_\Gamma^{1/2 - \varepsilon/2} \nabla^2 w \|_{L^2\left(\Omega \backslash \strip(\tilde c h)\right) }
\leq c
\| (\delta_\Gamma + h)^{1/2 - \varepsilon/2} \nabla^2 \widetilde w \|_{L^2(\Omega \backslash \strip(h) ) }, \quad \text{ for } x\in \Omega.
\]
Standard interior regularity~\eqref{eq:interior_regularity} yields, for a fixed $\rho>0$  and  any concentric balls of radius $r$ and $r(1+\rho)$, such that $B_{r(1+\rho)} \subset \Omega$, the estimate
 $\| \nabla^2 \widetilde w \|_{L^2(B_r) } \leq c r^{-1/2+\varepsilon/2} \| \widetilde w \|_{H^{3/2+\varepsilon/2}(B_{r(1+\rho)})}$ . A covering of $\Omega\backslash \strip(h)$ using balls of center $x_i$ and radii $r_i \sim h + \delta_\Gamma(x_i)$ shows
\[\| (\delta_\Gamma + h)^{1/2 - \varepsilon/2} \nabla^2 \widetilde w \|_{L^2(\Omega\backslash \strip(h)) } \leq c \| \widetilde w \|_{H^{3/2+\varepsilon/2}(\Omega)}. 
\]
Details on the Besicovitch covering theorem can be found in~\cite[Section 1.5.2]{evans:92} and~\cite[Chapter 5]{melenk:02}.

An analogue computation as in~\cite[Lemma 5.4]{wohlmuth:12}, where the case $\varepsilon = 0$ was considered, concludes the proof .
We bound the $K$-functional of the fractional Sobolev space $(H^1(\Omega), H^2(\Omega) )_{1/2+\varepsilon/2, 2} = H^{3/2+\varepsilon/2}(\Omega)$ by
\begin{align}\label{eq:K_Funk_Part1}
&\|\widetilde w \|_{H^{3/2+\varepsilon/2}(\Omega)}^2 = 
\int_{t=0}^{h} \left(t^{-1/2 - \varepsilon/2}K(t,\widetilde w) \right)^2 \frac{\mathrm{d}t}{t}  + \int_{t=h}^1 \left(t^{-1/2 - \varepsilon/2}K(t, \widetilde w) \right)^2 \frac{\mathrm{d}t}{t}\notag \\
&\quad\leq \int_{t=0}^{h} \left(t^{-1/2 - \varepsilon/2}K(t, \widetilde w) \right)^2 \frac{\mathrm{d}t}{t} +  \int_{t=h}^1 t^{-1-\varepsilon}  ~{\mathrm{d}t} ~ \sup_{t>0}\left(t^{-1/2} K(t, \widetilde w) \right)^2.
\end{align}
Again applying the interior regularity~\eqref{eq:interior_regularity}, we get $\| \widetilde w \|_{H^{2}(\Omega)} \leq c h^{-1/2} \|  w \|_{H^{3/2}(\Omega)}$ which yields $K(t, \widetilde w) \leq c t \|\widetilde w\|_{H^2(\Omega)} \leq c t h^{-1/2} \|  w \|_{H^{3/2}(\Omega)}$. Substituting this upper bound in the first integral of~\eqref{eq:K_Funk_Part1} and observing  $ \sup_{t>0}\left(t^{-1/2} K(t, \widetilde w) \right) \leq \|  w \|_{B_{2,\infty}^{3/2}(\Omega)}$, yields
\begin{align*}
\|\widetilde w \|_{H^{3/2+\varepsilon/2}(\Omega)}  &\leq c h^{-\varepsilon/2} \|  w \|_{B_{2,\infty}^{3/2}(\Omega)}.
\end{align*}

 Finally~\cite[Lemma 5.2]{wohlmuth:12} states $\|  w \|_{B_{2,\infty}^{3/2}(\Omega)} \leq c h^{1/2} \|v\|_{L^2(\Omega)}$ which concludes the proof. 
\end{proof}

Using local error estimates and the weighted regularity result proven above, we  show an approximation result for the Galerkin approximation of the dual problem in anisotropic norms.
\begin{lemma}\label{approximation_dirichlet_pb}
Given $v \in L^2(\Omega)$ with $\supp v \subset \overline{ \strip(h) }$,  consider $w = T^D v$ and the Galerkin approximation $w_h \in V_h\cap H_0^1(\Omega)$. 
 For $1<p = (1-\varepsilon)^{-1} \leq 2 $, the following approximation property holds:
\[
\| \nabla ( w - w_h) \|_{L(p,2)} \leq c h^{3/2-\varepsilon} \|v \|_{L^2(\Omega)}.
\]
\end{lemma}
\mybeginproof
We show the estimate on each patch $\Omega_i$, $i \in \{ 1,\ldots, N_\gamma\}$. 
In the definition of the norm, we decompose the integral in $\tau$ from $0$ to $D$ into two parts and find 
\begin{align*}
&\| \nabla(w - w_h)\|_{L(p,2;\Omega_i)}^p \\&\qquad=
 \int_{\tau = 0}^{\tilde c_1h} \| \nabla( w - w_h) \|_{L^2(\gamma_\tau)}^p ~\mathrm{d}\tau +  \int_{\tau = \tilde c_1 h}^D \| \nabla( w - w_h) \|_{L^2(\gamma_\tau)}^p~ \mathrm{d}\tau ,\notag
\end{align*}
where $\tilde c_1$ has to be adapted to  the constant $\tilde c$ resulting from the previous lemma.

The first term is an integral over a strip of width $\mathcal{O}(h)$.
The H\"older inequality with the exponents $2/p$, $2/(2-p)$ and the Fubini--Tonelli formula obviously yield for $p=(1-\varepsilon)^{-1}$
\begin{align*}
&\int_{\tau = 0}^{\tilde c_1 h} \| \nabla( w - w_h) \|_{L^2(\gamma_\tau)}^p ~ \mathrm{d}\tau \leq c h^{(2 - p)/2} \left( \int_{\tau = 0}^{\tilde c_1 h}  \| \nabla( w - w_h) \|_{L^2(\gamma_\tau)}^2  ~\mathrm{d}\tau \right)^{p/2} \notag \\
&\qquad \leq c h^{p(1/2 -  \varepsilon)}  \|\nabla(w - w_h)\|_{L^2(S(\tilde c_1 h))}^p.
\end{align*}
Since $\Omega$ is convex, we have $\|\nabla(w - w_h)\|_{L^2(\strip(\tilde c_1 h))} \leq \|\nabla(w - w_h)\|_{L^2(\Omega)} \leq c h \|v\|_{L^2(\Omega)}$, which gives
\begin{equation*}
\int_{\tau = 0}^{\tilde c_1 h} \| \nabla( w - w_h) \|_{L^2(\gamma_\tau)}^p \mathrm{d}\tau   \leq h^{p( 3/2 - \varepsilon)}\|v\|_{L^2(\Omega)}^p.
\end{equation*}

The second integral is estimated using a local approximation property and the regularity result given in Lemma~\ref{lem:dual_regularity}. First, we insert $\tau^{1/2}\tau^{-1/2}$ and use the H\"older inequality with the same exponents as before:
\begin{align*}
&\int_{\tau = \tilde c_1 h}^D \tau^{-1/2} \tau^{1/2} \| \nabla( w - w_h) \|_{L^2(\gamma_\tau)}^p ~\mathrm{d}\tau   
\\
&\qquad\leq
\left( \int_{\tau = \tilde c_1 h}^D \tau^{-1/(2-p)}  ~\mathrm{d}\tau \right)^{(2-p)/2}	%Multipliziert (da Hoelder)
\left( \int_{\tau = \tilde c_1 h}^D \tau^{1/p} \| \nabla( w - w_h) \|_{L^2(\gamma_\tau)}^2 ~\mathrm{d}\tau \right)^{p/2}  
\\&\qquad
\leq 
h^{- p \varepsilon/2} \| \tau^{1/2 - \varepsilon/2} \nabla(w - w_h) \|_{L^2(\Omega\backslash \strip(\tilde c_1 h))}^p.
\end{align*}
Based on the discussion in~\cite[Section 5.1.2]{wohlmuth:12}, we derive the bound
\begin{align}
 &\| \tau^{1/2 - \varepsilon/2} \nabla(w - w_h) \|_{L^2(\Omega\backslash \strip(\tilde c_1 h))}  \notag \\
  &\qquad \leq \| \tau^{1/2 - \varepsilon/2} \nabla(w - I_h w) \|_{L^2(\Omega\backslash \strip(\tilde c_2h))} +  \| \tau^{-1/2 - \varepsilon/2}( w - w_h )\|_{L^2(\Omega\backslash \strip({\tilde c_2 h}))}
   \label{eq:result_wahlbin}
\end{align}
for an arbitrary but fixed $\tilde c_2$, if $\tilde c_1$ is chosen sufficiently large.
This estimate is based on local approximation properties found in~\cite{wahlbin:91, wahlbin:95} and a Besicovitch covering argument.

To estimate the first term, we exploit the regularity result derived in Lemma~\ref{lem:dual_regularity}. Based on $\tilde c$, which is given from the previous lemma, we can choose $\tilde c_2$ and $\tilde c_1$ sufficiently large, such that
\begin{align*}
 \| \tau^{1/2 - \varepsilon/2} \nabla(w - I_h w) \|_{L^2(\Omega\backslash \strip(\tilde c_2h))} 
 & \leq ch \| \tau^{1/2 - \varepsilon/2} \nabla^2 w \|_{L^2(\Omega\backslash \strip(\tilde ch))} \\
 & \leq c h^{3/2-\varepsilon/2} \|v\|_{L^2(\Omega)}.
\end{align*}

Using the convexity of $\Omega$ the second term of~\eqref{eq:result_wahlbin} can be bounded easily by
\begin{align*}
\| \tau^{-1/2 - \varepsilon/2} (w - w_h )\|_{L^2(\Omega\backslash \strip(\tilde c_2 h))} 
	& \leq  h^{-1/2 - \varepsilon/2} \| w - w_h \|_{L^2(\Omega)} \\&\leq h^{3/2 - \varepsilon/2}  \|v\|_{L^2(\Omega)}.  \myendproof
\end{align*}

The previously shown bounds in anisotropic norms are sufficient to show primal estimates in a neighbourhood of the boundary. 
For a final bound of the Lagrange multiplier, we also need to consider a dual problem with Neumann data, as defined in~\cite[Section 5.2]{wohlmuth:12}. Given $v\in L^2(\Omega)$, $\supp v \subset \overline{\strip(h)}$, define $w_v^N$ such that
\begin{equation} \label{eq:dual_neumann}
-\Delta w_v^N = v - \frac{1}{\Omega} \int_{\Omega} v~\mathrm{d}x  ~ \text{ in } \Omega, \quad \partial_n w_v^N = 0,~ \text{ on } \Gamma, \quad \int_\Omega w_v^N~\mathrm{d}x = 0. 
\end{equation}

Denote by $V_h^{-1}$ the space of discrete functions without any restriction of the boundary values. Using the same arguments as before, we can adapt the proof of~\cite[Lemma 5.7]{wohlmuth:12} and show the following statement based on the dual Neumann problem.% \marginpar{duales Neumann PB trotzdem definieren?}
\begin{corollary}\label{cor:dual_neumann}
Let $u \in V \cap H^{5/2-\varepsilon}(\Omega)$ and $u_h^N  \in V_h^{-1}$ satisfy the orthogonality condition
$a(u-u_h^N , v_h) = 0$ for  $v_h \in V_h^{-1}$ and $\int_{\strip(h)} u - u_h^N ~\mathrm{d}x = 0$,
then
\begin{align*}
 \| u -u_h^N  \|_{L^2(\strip(h))} &\leq c h^{5/2-\varepsilon} \|u\|_{H^{5/2-\varepsilon}(\Omega)},\\
\left| u - u_h^N  \right|_{H^{1/2}(\Gamma)} &\leq  c h^{3/2-\varepsilon} \|u\|_{H^{5/2-\varepsilon}(\Omega)}.
\end{align*}
\end{corollary}

\subsection{Error bound for the Dirichlet--Neumann map} \label{sec:main_proof}
With the results of the previous subsection, we can estimate the $H^{-1/2}(\Gamma_S)$ error of the Dirichlet--Neumann map $\Nf - S ( \left. \Tr u \right|_{\Gamma_S} )$ and the mesh-dependent Dirichlet--Neumann map $\Nfh - S_h ( \left. \Tr u \right|_{\Gamma_S} )$, see Section~\ref{subsec:reformulation}, in two steps. 
This bound is the last step to show the primal estimate in Theorem~\ref{thm:main_result}. %The steps are analogue to~\cite[Theorem 2.1, 2.5]{wohlmuth:12}. 
Firstly, we relate the error of the dual variable to the error of the primal variable in a small strip around $\Gamma$ using the dual Neumann problem~\eqref{eq:dual_neumann}. Secondly, the error in the strip is estimated using the dual Dirichlet problem~\eqref{dual_dirichlet} and the approximation results derived in the anisotropic norms.

\begin{theorem}\label{thm:estimate:steklov} 
Assuming the solution $u$ of the Signorini problem~\eqref{vi_cont} to be in $H^{5/2-\varepsilon}(\Omega)$, then it holds
\[
\| \lambda - \widetilde \lambda_h  \|_{H^{-1/2}(\Gamma_S)} 
 \leq ch^{3/2 - \varepsilon} \| u \|_{H^{5/2-\varepsilon}(\Omega)}.
\]
\end{theorem}
\begin{proof} 
The proof is divided into two steps. Firstly, we bound the dual error by the primal error on a small neighbourhood of the boundary. Secondly, we bound the primal error on a small strip using the anisotropic estimates stated in Lemma~\ref{sz:approximation_b} and~\ref{approximation_dirichlet_pb}.

To be more precise, the first step is to show the upper bound
\begin{align} \label{eq:dual_estimate_general}
\| \lambda - \widetilde \lambda_h  \|_{H^{-1/2}(\Gamma_S)} 
 \leq ch^{3/2 - \varepsilon} \| u \|_{H^{5/2-\varepsilon}(\Omega)} + c \frac1h\| u - \widetilde u_h \|_{L^2(\strip(h))}.
\end{align}
We use the saddle point formulation to represent the dual error by discrete harmonic functions on the domain. Using the stability of the harmonic extension and an inverse trace inequality, we can relate the dual error to the primal error on the strip $S(h)$.

We start using the uniform inf-sup stability in the $H^{-1/2}(\Gamma_S)$ norm to get
\begin{align*}%\label{eq:dual_estimate_eq1}
\| \lambda - \widetilde \lambda_h \|_{H^{-1/2}(\Gamma_S)} 
& \leq c  \inf_{ \mu_h \in M_h } \| \lambda - \mu_h \|_{ H^{-1/2}(\Gamma_S) } + c \sup_{z_h \in W_h} \frac{ \langle z_h, \lambda - \widetilde \lambda_h \rangle_{\Gamma_S} }{ \|z_h\|_{H^{1/2}_{00}(\Gamma_S) } } \notag \\
& \leq c h^{3/2-\varepsilon} \| \lambda \|_{H^{1-\varepsilon}(\Gamma_S)} + c \sup_{z_h \in W_h} \frac{ a(\widetilde {u}_h -  u, \mathcal{E}_h z_h) }{ \|z_h\|_{H^{1/2}_{00}(\Gamma_S) } },
\end{align*}
%where $\mathcal{E}_h w_h \in V_h$ is the discrete harmonic extension of $w_h \in W_h$.
where $\mathcal{E}_h z_h \in V_h$ is the discrete harmonic extension of $z_h^\Gamma \in H^{1/2}(\Gamma)$ which is the trivial extension to $\Gamma$ of $z_h \in W_h \subset H^{1/2}_{00}(\Gamma_S)$.

We replace $u$ by a discrete function $u_h^N \in V_h^{-1}$ satisfying the requirements of Corollary~\ref{cor:dual_neumann}. We also use the fact that $\mathcal{E}_h z_h $ and $\widetilde u_h - u_h^N$ are discrete harmonic to see
\begin{align*}%\label{eq:dual_estimate_eq2}
\sup_{z_h \in W_h} \frac{ a(\widetilde{u}_h - u, \mathcal{E}_h z_h) }{ \|z_h\|_{H^{1/2}_{00}(\Gamma_S) } }  
& = \sup_{z_h \in W_h} \frac{ a(\widetilde{u}_h - u_h^N, \mathcal{E}_h z_h) }{ \|z_h\|_{H^{1/2}_{00}(\Gamma_S) } } \leq c\left|\widetilde u_h - u_h^N\right|_{H^{1/2}(\Gamma)}.
\end{align*}
Using an inverse inequality, we get
\begin{align*}
\left|\widetilde u_h - u_h^N\right|_{H^{1/2}(\Gamma)}
& \leq  c \frac1h \|\widetilde u_h - u_h^N\|_{L^2(\strip(h))} \notag \\&\leq c \frac1h \|u - u_h^N \|_{L^2(\strip(h))} +  c \frac1h \|u - \widetilde u_h\|_{L^2(\strip(h))}.
\end{align*}
Now Corollary~\ref{cor:dual_neumann} results in~\eqref{eq:dual_estimate_general}.

To bound
$\| u - \widetilde u_h \|_{L^2(\strip(h))}  
	%\leq c h^{5/2-\varepsilon}  \| u \|_{H^{5/2-\varepsilon}(\Omega)},
$
we employ different Galerkin orthogonalities to get a  suitable representation of the error in the whole domain based on the solution of the dual problem.
Applying Green's formula, we obtain the representation of the local error $e_h := u - \widetilde u_h$:
\begin{align*}
\|e_h\|_{L^2(\strip(h))} 
 &= \sup_{\|v\|_{L^2(\strip(h))} = 1} (e_h, v)_{L^2(\Omega)}      
  =\sup_{ \|v\|_{L^2(\strip(h))} = 1} (e_h, -\Delta (T^Dv))_{L^2(\Omega)} \\
 &=  \sup_{ \|v\|_{L^2(\strip(h))} = 1} a(T^Dv, e_h) - \langle e_h, \partial_n (T^Dv) \rangle_{\Gamma_S},
\end{align*}
where $T^Dv \in H_0^1(\Omega)$ is the solution to the dual problem~\eqref{dual_dirichlet}.

Let us introduce the conforming finite element approximation of $w := T^Dv$ as $ w_h \in V_h \cap H_0^1(\Omega)$, and denote $\lambda_w := - \left. \partial_n w\right|_{\Gamma_S}$.  We recall the following  orthogonality results: Using the Galerkin orthogonality in the domain for the variational inequality~\eqref{vi_discrete}, it holds $
a(w_h, e_h) = 0$, since $\Trace w_h = 0$. 
 We recall, that the definition of the Scott--Zhang operator $P_h$, see Section~\ref{subsec:anisotropic}, guarantees  $\langle u - P_h u, \mu_h \rangle_{\Gamma_S} = 0$ as well as $\langle P_h u - \widetilde u _h, \mu_h \rangle_{\Gamma_S} = 0$ for $\mu_h \in M_h$.
We can then conclude with
\begin{align*}
a(w - w_h, P_h u - \widetilde u_h) + \langle P_h u - \widetilde u_h, \lambda_w\rangle_{\Gamma_S} &= 0.
\end{align*}
%\end{subequations}
For $1/p + 1/q = 1$, we find using the terms discussed above
\begin{align*}
&a(w, e_h) + \langle e_h, \lambda_w \rangle_{\Gamma_S} 
 = a(w -w_h, u - P_h u) + \inf_{\mu_h \in M_h} \langle u - P_h u, \lambda_w - \mu_h \rangle_{\Gamma_S} \notag \\
 & \quad \leq \| \nabla( w -w_h)\|_{L(p,2)} \| \nabla( u - P_h u) \|_{L(q,2)}  \notag \\
 & \qquad + \| u - P_h u\|_{L^2(\Gamma_S)}  \inf_{\mu_h \in M_h}\|\lambda_w - \mu_h \|_{L^2(\Gamma_S)}.
\end{align*}
The convexity of $\Omega$ guarantees $\lambda_w \in H^{1/2}(\Gamma_S)$ with $\|\lambda_w\|_{H^{1/2}(\Gamma_S)} \leq c \|v\|_{L^2(\Omega)}$. Setting $q = \varepsilon^{-1}, p = (1-\varepsilon)^{-1}$, the best approximation of the dual space, Equation~\eqref{sz:approximation_a_eq} and Lemmas~\ref{sz:approximation_b} and~\ref{approximation_dirichlet_pb}  yield the result.
\end{proof}

Summarizing the results of Lemmas~\ref{lem:strang},~\ref{lem:estimate_unperturbed} and Theorem \ref{thm:estimate:steklov} shows the a priori result for the primal variable of Theorem~\ref{thm:main_result}.

\subsection{An improved result on the $L^2(\Omega)$ error} \label{sec:L2_improved}
Based on~\cite{natterer:76} a convergence order $h^{3/2}$ in the $L^2(\Omega)$ norm was stated in~\cite{hild:01}. However the required $H^2(\Omega)$ regularity of the dual problem is very strong, since the dual problem is a variational inequality. 
Based on the improved trace estimate, we can show almost the same order without involving a dual inequality problem.
\begin{corollary}
Let $u$ be the solution of~\eqref{vi_cont} and $u_h$ be the solution of~\eqref{vi_discrete}. Assuming $u\in H^{5/2-\varepsilon}(\Omega)$ and that Assumption~\eqref{eq:fractal_assumption} holds for $d=3$, then we get
\[
\|u - u_h\|_{L^2(\Omega)} \leq c h^{3/2-\varepsilon} \|u\|_{H^{5/2-\varepsilon}(\Omega)}.
\]
\end{corollary}
\mybeginproof
The proof is based on an Aubin--Nitsche type argument using a linear dual problem with homogeneous Dirichlet conditions.  Due to the nonlinear Signorini condition, an additional error term on $\Gamma_S$ needs to be bounded.

Let $w\in H^1_0(\Omega)$ solve 
%\[a(u,v) =  \int_\Omega v (u-u_h) \mathrm{d}x, \quad v\in H_0^1(\Omega).
%\]
$
-\Delta w = u - u_h%, \text{ in } \Omega, w = 0, \text{ on } \Gamma.
$ in $\Omega$.
Since $\Omega$ is convex, it holds $\|w \|_{H^2(\Omega)}\leq c \|u-u_h\|_{L^2(\Omega)}$ and $\| \partial_n w_v \|_{L^2(\Gamma_S)} \leq \|w \|_{H^{2}(\Omega)} $. Applying 
Green's formula yields
\begin{align*}
\|u - u_h\|_{L^2(\Omega)}^2 &= \int_\Omega \nabla  w^{\rm T} \nabla (u-u_h) ~\mathrm{d}x -   \langle u - u_h , \partial_n w_v \rangle_{\Gamma_S}.
\end{align*}
The first term can be bounded as it is standard in Aubin--Nitsche arguments, due to the homogeneous Dirichlet values of $w$.
For the second term, we use the trace estimate provided in Theorem~\ref{thm:main_result}:
\begin{align*}
\langle u - u_h , \partial_n w \rangle_{\Gamma_S} &\leq \|u - u_h\|_{L^{2}(\Gamma_S)}  \|\partial_n w\|_{L^2(\Gamma_S)} \\&\leq c h^{3/2-\varepsilon} \|u\|_{H^{5/2-\varepsilon}(\Omega)} \|w\|_{H^{2}(\Omega)}. \quad \myendproof
\end{align*}
\begin{remark}
We note that in the proof of the $L^2(\Omega)$ norm we use the trivial bound $\|u - u_h\|_{L^2(\Gamma_S)}\leq\|u - u_h\|_{H^{1/2}(\Gamma_S)}$. Thus an extra $h^{1/2}$ would be possibly gained, if a higher order $L^2(\Gamma_S)$ bound was available.	
\end{remark}

%% file: 06_LM_approx.tex
The $H^{-1/2}(\Gamma_S)$ norm for the Lagrange multiplier of the Signorini problem arising in the saddle point formulation~\eqref{spp_cont} can be estimated using similar arguments as those used in Theorem~\ref{thm:estimate:steklov}.  Due to the given primal estimate of Theorem~\ref{thm:main_result}, no estimate on a strip is needed here.
By standard techniques, the $L^2(\Gamma_S)$ norm can also be estimated.
\begin{theorem}\label{thm:lagrange_estimates}
Let $(u, \lambda)$ be the solution of the saddle point formulation~\eqref{spp_cont}. If the regularity requirement $u\in H^{5/2-\varepsilon}(\Omega)$ and for $d=3$ the assumption~\eqref{eq:fractal_assumption} hold, then
	\begin{align*}
		\| \lambda - \lambda_h \|_{H^{-1/2}(\Gamma_S)} &\leq c  h^{3/2 - \varepsilon} \|u\|_{H^{5/2-\varepsilon}(\Omega)},  \\
		\| \lambda - \lambda_h \|_{L^2(\Gamma_S)} &\leq c  h^{1 - \varepsilon} \|u\|_{H^{5/2-\varepsilon}(\Omega)}. 
	\end{align*}
\end{theorem}
\mybeginproof
The first line of the saddle point problem~\eqref{spp_cont_line_1} and its Galerkin discretization~\eqref{spp_discrete_line_1} yield $a(u - u_h, v_h) + \langle v_h, \lambda - \lambda_h \rangle_{\Gamma_S} = 0$ for $v_h \in V_h$. 
Similar arguments as in the proof of Theorem~\ref{thm:estimate:steklov} give
\begin{align*} 
\| \lambda -  \lambda_h \|_{H^{-1/2}(\Gamma_S)} 
	&\leq c h^{3/2-\varepsilon} \| \lambda \|_{H^{1-\varepsilon}(\Gamma_S)} + c \sup_{z_h \in W_h} \frac{ a({u}_h - u, \mathcal{E}_h z_h) }{ \|z_h\|_{H^{1/2}_{00}(\Gamma_S) } }  \notag \\
	&\leq c h^{3/2-\varepsilon} \| \lambda \|_{H^{1-\varepsilon}(\Gamma_S)} + c\left|u_h - u_h^N\right|_{H^{1/2}(\Gamma)},
\end{align*}
where we exploit the fact that $a(u, \mathcal{E}_h z_h)=a(u_h^N, \mathcal{E}_h z_h)$ and a stability estimate for discrete harmonic functions, see~\cite[Lemma 4.10]{toselli:05}.
It is important to note, that $u_h^N \in V_h^{-1}$, which is defined as in the proof of Theorem~\ref{thm:estimate:steklov}, only depends on $u$, not on $u_h$ or $\widetilde u_h$. Nevertheless $u_h - u_h^N$ is discrete harmonic due to the Galerkin approximation of the saddle point problem. Using Corollary~\ref{cor:dual_neumann} and the primal estimate of Theorem~\ref{thm:main_result}, we conclude
\begin{align*}
\left|u_h - u_h^N\right|_{H^{1/2}(\Gamma)} 
	&\leq \left|u - u_h\right|_{H^{1/2}_{00}(\Gamma_S)} +  \left|u - u_h^N\right|_{H^{1/2}(\Gamma)} \notag \\
	&\leq c h^{3/2-\varepsilon}  \| u \|_{H^{5/2-\varepsilon}(\Omega)}.
\end{align*}

The remaining error estimate in the $L^2(\Gamma_S)$ norm  follows by an inverse inequality and the best approximation properties:
\begin{align*}
&\| \lambda - \lambda_h \|_{L^2(\Gamma_S)} \leq \inf_{\mu_h \in M_h} \left( \|  \lambda - \mu_h \|_{L^2(\Gamma_S) } + \| \mu_h - \lambda_h\|_{L^2(\Gamma_S)} \right)  \notag \\
&\quad \leq 
c \inf_{\mu_h \in M_h} \left( \|  \lambda - \mu_h \|_{L^2(\Gamma_S) } + \frac{1}{\sqrt{h}} \| \mu_h - \lambda\|_{H^{-1/2}(\Gamma_S)}\right) \\&\qquad + \frac{c}{\sqrt{h}} \| \lambda - \lambda_h\|_{H^{-1/2}(\Gamma_S)}. \myendproof
\end{align*}

%% file: 07_numerical_results.tex
We chose an example with an analytically known solution on $\Omega = (0,1.4+e/2.7)\times(0,0.5)$, where $\Gamma_S = (0,1.4+e/2.7)\times \{0\}$. The  choice  of the domain was done in order to have an easy representation of the solution with an asymmetry over the Signorini boundary. We chose the volumetric and Dirichlet boundary data as well as the initial gap $g(x) = 0$ according to the exact solution which is constructed as follows.

In polar coordinates, the singular component (see also Remark~\ref{remark_regularity}) is given by $u_{\rm sing}(r,\theta) = r^{3/2}\sin(3/2 ~\theta)$ which we will also denote $u_{\rm sing}(x,y)$ in Cartesian coordinates. As this singular component has a one-sided active area, we need to modify the function to ensure the condition that the active set $\Gamma^{\mathrm{act} }$ is a compact subset of the Signorini boundary $\Gamma_S$.
The singular function is translated such that the transmission point between the active and inactive part is at $x_l \coloneq 0.2 + 0.3 / \pi \approx 0.295$. A spline of polynomial order four is used as a cut-off function $u_{\rm cut}$.  Adding a weighted reflection of this function, we get a function with a compact contact area. The second transmission point is set to  $x_r \coloneq 1.2 -  0.3 / \pi \approx 1.105$.
For some scalar weight $a > 0$ (in the examples $a = 0.7$), the solution is given by
\begin{align*}
u(x,y)  \coloneq  \big( u_{\rm sing}(x-x_l,y)~ u_{\rm cut}(x)  + a  ~ u_{\rm sing}(x_r - x,y)~u_{\rm cut}(1.4 - x)\big)(1 - y^2 ).
\end{align*}

For the right hand side $f \coloneq -\Delta u$, the Dirichlet data $u_D \coloneq \left. u \right|_{\Gamma_D}$ and $g(x)=0$, the solution satisfies the Signorini-type problem~\eqref{vi_cont}. The actual contact area is given by $\Gamma^{\mathrm{act}} = [0.2 + 0.3/\pi, 1.2 - 0.3/\pi]$. This choice of the contact area was made to ensure, that no vertex of the mesh coincides with its boundary. The domain yields an asymmetry of the contact area. The desired regularity $u \in H^{5/2-\varepsilon}(\Omega)$ is given by construction. We start from a coarse, quadrilateral, initial mesh of $4\times2$ elements and refine uniformly.

\begin{figure}[!htbp]
\centering
\includegraphics[width=0.49\textwidth]{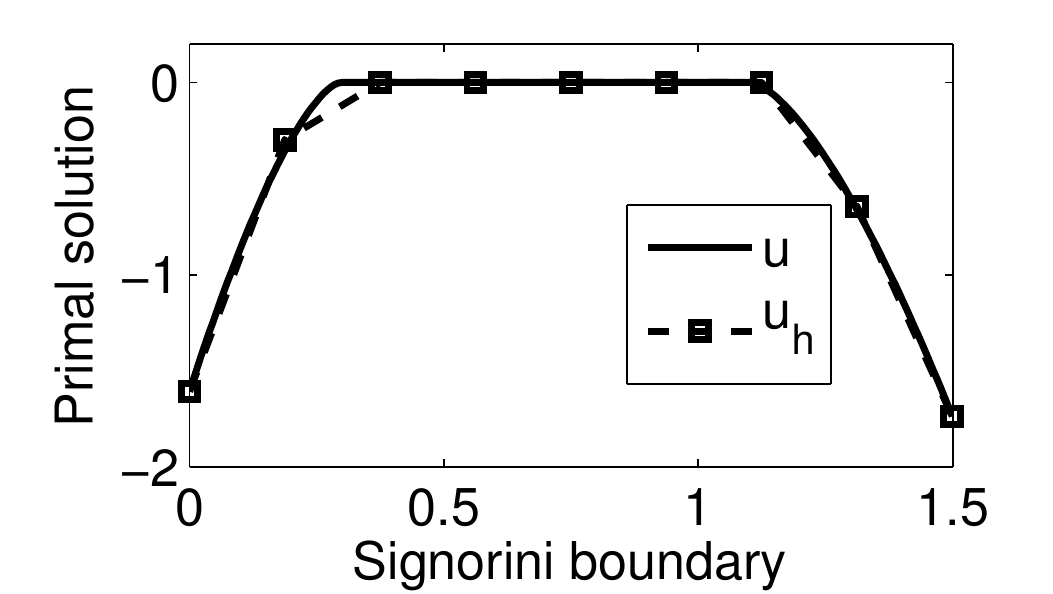}	%
\includegraphics[width=0.49\textwidth]{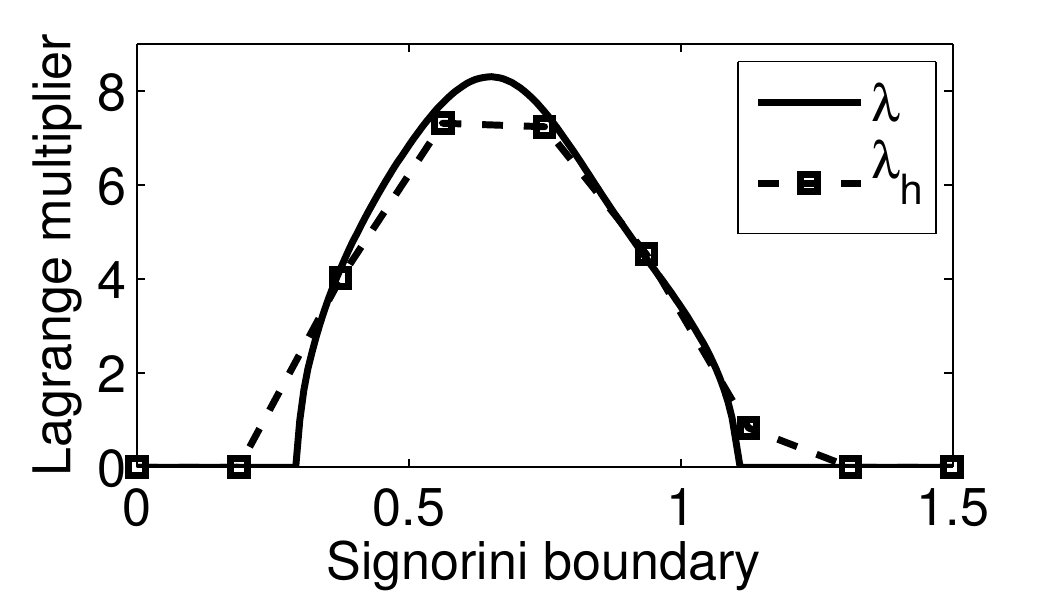}
\caption{   Exact solution and finite element approximation on level 2. Values for the primal solution (left) and the dual solution (right).  }
\label{fig:solution_level2}
\end{figure}

The exact solution on the Signorini boundary as well as a coarse finite element approximation are displayed in Figure~\ref{fig:solution_level2}.
In Figure~\ref{fig:errors_level6}, the error distribution restricted to $\Gamma_S$ on a fine finite element grid is shown. Since the discrete Lagrange multiplier is based on a biorthogonal basis and hence is discontinuous, a post-processing is applied for the visualization and the error computation. Instead of $\lambda_h = \sum_{i=1}^{N_{M_h}} \lambda_i \psi_i \in M_h$, we represent the Lagrange multiplier as
\[
\widehat \lambda_h = \sum_{i=1}^{N_{M_h}} \lambda_i \varphi_i \in W_h.
\]
As it was shown in~\cite[Section 3.3]{hueeber:08}, the order of convergence of $\widehat \lambda_h$ is the same as for $\lambda_h$.  Although the proof was shown for rates up to the order $h$, an analogue proof can be performed for the current situation.

\begin{figure}[!htbp]
\centering
\includegraphics[width=0.49\textwidth]{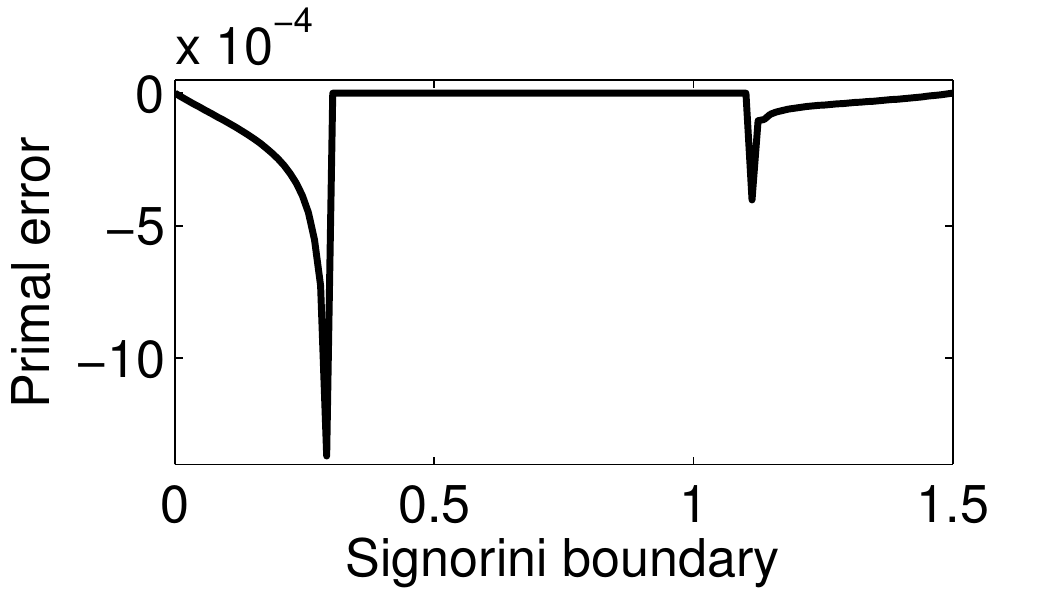}
\includegraphics[width=0.49\textwidth]{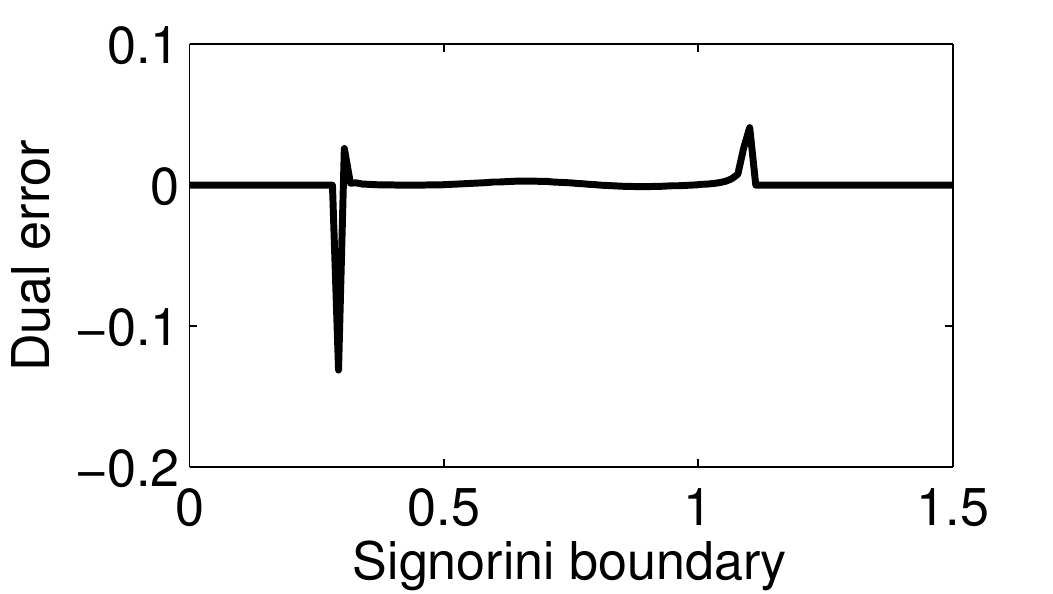}
\caption{  Discretization error displayed at the Signorini boundary. Error of the primal variable (left) and of the dual variable (right) at level $k = 6$.  }
\label{fig:errors_level6}
\end{figure}

The error distribution reflects the singularities of the solution at $\partial \Gammaact$. We observe two peaks of large errors at the boundary of the active set caused by the reduced regularity at these points. The error in the interior of the domain is of a similar order, hence the overall error is not dominated by the error on the boundary.

\begin{table}[!htbp]
\centering
\caption{ Relative errors of the primal and dual solution at different mesh levels $k$ and an averaged numerical convergence order.}
{\small
\begin{tabular}[h!]{ |c||c|c||c|c||c|c|}
\hline
$k$ & 
	\multicolumn{2}{|c||}{$\|\lambda - \lambda_h \|_{L^2(\Gamma_S)} $} & 
	\multicolumn{2}{|c||}{$\|u - u_h \|_{L^2(\Gamma_S)}$}
	& \multicolumn{2}{|c|}{$\|u - u_h \|_{L^2(\Omega)}$} 
	\\
\hline
\hline 
$1$  & $3.2629e\!-\!01$ & $-$  & $2.1724e\!-\!01$ &   & $1.0050e\!-\!01$ & \\
$2$ & $1.2955e\!-\!01$ &  $1.33$ & $4.2717e\!-\!02$  &  $2.35$ & $2.5761e\!-\!02$ &  $1.96$ \\
$3$ & $4.4331e\!-\!02$ &  $1.44$ & $7.0192e\!-\!03$  &  $2.48$ & $6.2041e\!-\!03$ &  $2.01$ \\
$4$ & $1.8560e\!-\!02$ &  $1.38$ & $2.3468e\!-\!03$  &  $2.18$ & $1.5550e\!-\!03$ &  $2.00$ \\
$5$ & $1.5159e\!-\!02$ &  $1.11$ & $9.3812e\!-\!04$  &  $1.96$ & $4.0559e\!-\!04$ &  $1.99$ \\
$6$ & $5.8243e\!-\!03$ &  $1.16$ & $1.8083e\!-\!04$  &  $2.05$ & $9.8738e\!-\!05$ &  $2.00$ \\
$7$ & $2.7746e\!-\!03$ &  $1.15$ & $4.0096e\!-\!05$  &  $2.07$ & $2.4336e\!-\!05$ &  $2.00$ \\
$8$  & $1.9410e\!-\!03$ &  $1.06$ & $1.7967e\!-\!05$  &  $1.94$ & $6.4165e\!-\!06$ &  $1.99$ \\
$9$ & $9.8497e\!-\!04$ &  $1.05$ & $4.6480e\!-\!06$  &  $1.94$ & $1.5986e\!-\!06$ &  $1.99$ \\
$10$  & $4.0873e\!-\!04$ &  $1.07$ & $9.7558e\!-\!07$  &  $1.97$ & $3.8972e\!-\!07$ &  $2.00$ \\
$11$ & $1.7042e\!-\!04$ &  $1.09$ & $1.9775e\!-\!07$  &  $2.01$ & $9.5737e\!-\!08$ &  $2.00$ \\
\hline
\end{tabular}}
\label{tab:errors_rate}
\end{table}

In Table~\ref{tab:errors_rate}, the computed $L^2$ norms of the error as well as the estimated rate of convergence are depicted for each level $k$. Errors in fractional Sobolev norms are given in Figure~\ref{fig:errors_fractional_norms}. 
The $L^2(\Gamma_S)$, $L^2(\Omega)$ and $H^1(\Omega)$ norms were computed by an adaptive integration to guarantee reliable results for the nonsmooth solution. 
The dual norm $H^{-1}(\Gamma)$, was estimated as the norm of the dual space to a fine finite element space. To be more precise, the Lagrange multiplier on each level $k=1,\ldots,11$ was prolongated up to level 15. On this level, we replace $\lambda$ by the piecewise linear interpolation and compute $\widehat{\lambda}_{h_k} - \mathcal{I}_{15}\lambda \in W_{h_{15}}$. Note that we have $\widehat{\lambda}_h \in W_h$ due to the post-processing as described above. The $H^{-1}(\Gamma_S)$ norm is approximated by the dual norm of $W_{h_{15}}$, i.e.
\[
\|  \lambda - \lambda_{h_k} \|_{H^{-1}(\Gamma_S)} \approx \sup_{w_{h_{15}} \in W_{h_{15}}} \frac{\int_{\Gamma_S} ( \widehat{\lambda}_{h_k} - \mathcal{I}_{15}\lambda) w_{h_{15}} \mathrm{d}x}{\|w_{h_{15}} \|_{H^1(\Gamma_S)}}.
\]
The fractional order Sobolev spaces $H^{1/2}_{00}(\Gamma_S)$ and $H^{-1/2}(\Gamma_S)$ were bounded using their interpolation property, i.e.,
\begin{align*}
\|v\|_{H^{1/2}_{00}(\Gamma_S)} \approx \|v\|_{H^1(\Gamma_S)}^{1/2} \|v\|_{L^2(\Gamma_S)}^{1/2},
\qquad 
\|v\|_{H^{-1/2}(\Gamma_S)} \approx \|v\|_{H^{-1}(\Gamma_S)}^{1/2} \|v\|_{L^2(\Gamma_S)}^{1/2}.
\end{align*}

\begin{figure}[!htbp]
\centering
\includegraphics[width=0.49\textwidth]{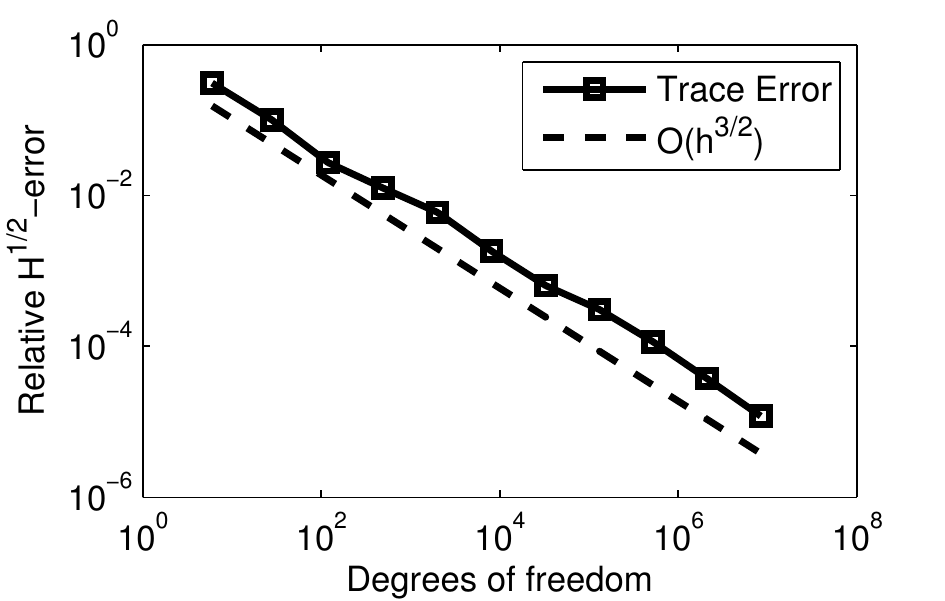}	%
\includegraphics[width=0.49\textwidth]{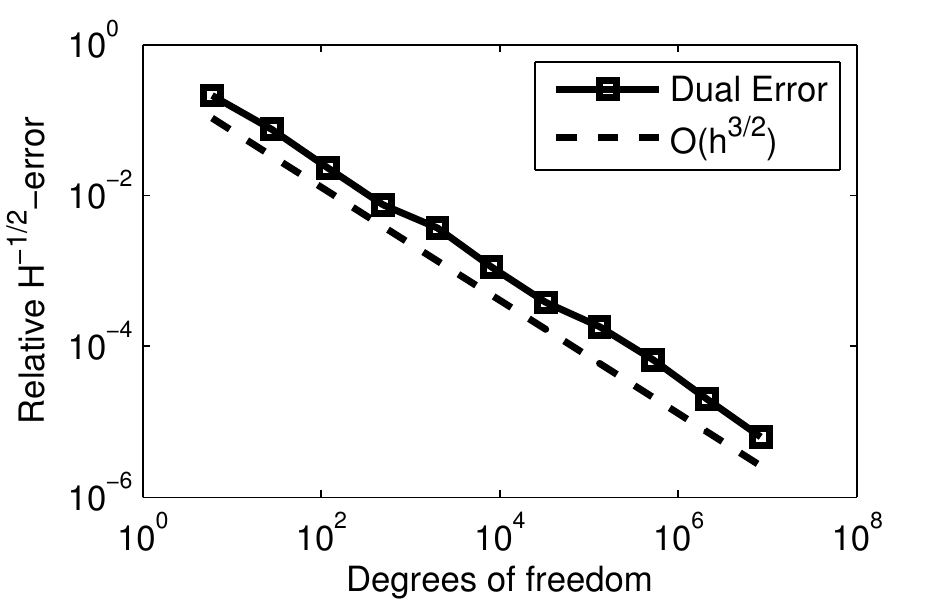}
\caption{   Estimated convergence rates in fractional Sobolev spaces. Left: $H^{1/2}(\Gamma_S)$ norm for the primal solution. Right: $H^{-1/2}(\Gamma_S)$ norm for the dual solution..  }
\label{fig:errors_fractional_norms}
\end{figure}

The averaged convergence rates $\alpha_k$ as given in Table~\ref{tab:errors_rate}, were computed in comparison to the first solution, by the formula
\begin{equation*}
\left( \frac{\text{err}_1}{\text{err}_k}\right) = \left(\frac{1}{2} \right)^{\alpha_k (k - 1)}.
\end{equation*}
We observe optimal order convergence rates in the $L^2$ norms, which is as expected from our theory for the Lagrange multiplier, whereas for the $L^2(\Omega)$ and the $L^2(\Gamma_S)$ norm we obtain better rates, than given by the theory. %\marginpar{neue bemerkung}
A closer look reveals, that the convergence rates from level to level for the values on $\Gamma_S$ vary more strongly. 
This is related to the fact, that the discrete resolution of the active set is restricted to the vertices of the finite element mesh. Depending on the quality of the approximation of the active set, the rates for values on $\Gamma_S$ can be larger or smaller than expected.
In Figure~\ref{fig:errors_fractional_norms}, we see that the averaging described above is a reasonable estimate for the convergence rate. 

\begin{table}[!htbp]
\centering
\caption{ Distance of the transmission points $x_{l}$ and $x_r$ to the discrete transmission points $x_{l, h}$ and $x_{r,h}$ on level $k$, compared to the mesh size $h$}
{\small
\begin{tabular}[h!]{ |l||c|c|c|c|}
\hline
$k$  & 
	$|x_l - x_{l, h} |$  &  $|x_l - x_{l, h} | / h$& 
	$|x_r - x_{r, h} |$  &  $|x_r - x_{r, h} | / h$
	\\
\hline
\hline 
$1$& $7.9339e\!-\!02$  & $0.21$ & $1.9989e\!-\!02$  & $0.05$  \\
$2$& $7.9339e\!-\!02$  & $0.42$ & $1.9989e\!-\!02$  & $0.11$  \\
$3$& $1.4369e\!-\!02$  & $0.15$ & $1.9989e\!-\!02$  & $0.21$  \\
$4$& $1.4369e\!-\!02$  & $0.31$ & $2.6865e\!-\!02$  & $0.57$  \\
$5$& $9.0579e\!-\!03$  & $0.39$ & $3.4384e\!-\!03$  & $0.15$  \\
$6$& $2.6556e\!-\!03$  & $0.23$ & $3.4384e\!-\!03$  & $0.29$  \\
$7$& $3.2012e\!-\!03$  & $0.55$ & $3.4384e\!-\!03$  & $0.59$  \\
$8$& $2.7280e\!-\!04$  & $0.09$ & $5.1006e\!-\!04$  & $0.17$  \\
$9$& $2.7280e\!-\!04$  & $0.19$ & $5.1006e\!-\!04$  & $0.35$  \\
$10$& $2.7280e\!-\!04$  & $0.37$ & $2.2203e\!-\!04$  & $0.30$  \\
$11$& $9.3246e\!-\!05$  & $0.25$ & $1.4402e\!-\!04$  & $0.39$  \\
\hline
\end{tabular}}
\label{tab:active_distance}
\end{table}

We are also interested in a good resolution of the actual contact set, so we take a closer look at the solution near the boundary of $\Gamma^{\mathrm{act}}$. The discrete active set is taken as the coincidence set of the primal solution $\Gamma^{\mathrm{act}}_h \coloneq \{x \in \Gamma_S: u_h(x) = 0 \}$.  In Table~\ref{tab:active_distance}, the distance between the transmission points and the discrete transmission points is shown and compared to the mesh size. We note, that the distance is always smaller than the mesh size.  Since no vertex matches with a transmission point, this is the best we can expect.
Figure~\ref{fig:transmission_pt} shows the dual solution and some finite element approximations on the Signorini boundary.
\begin{figure}[!htbp]
\centering
\includegraphics[width=0.7\textwidth]{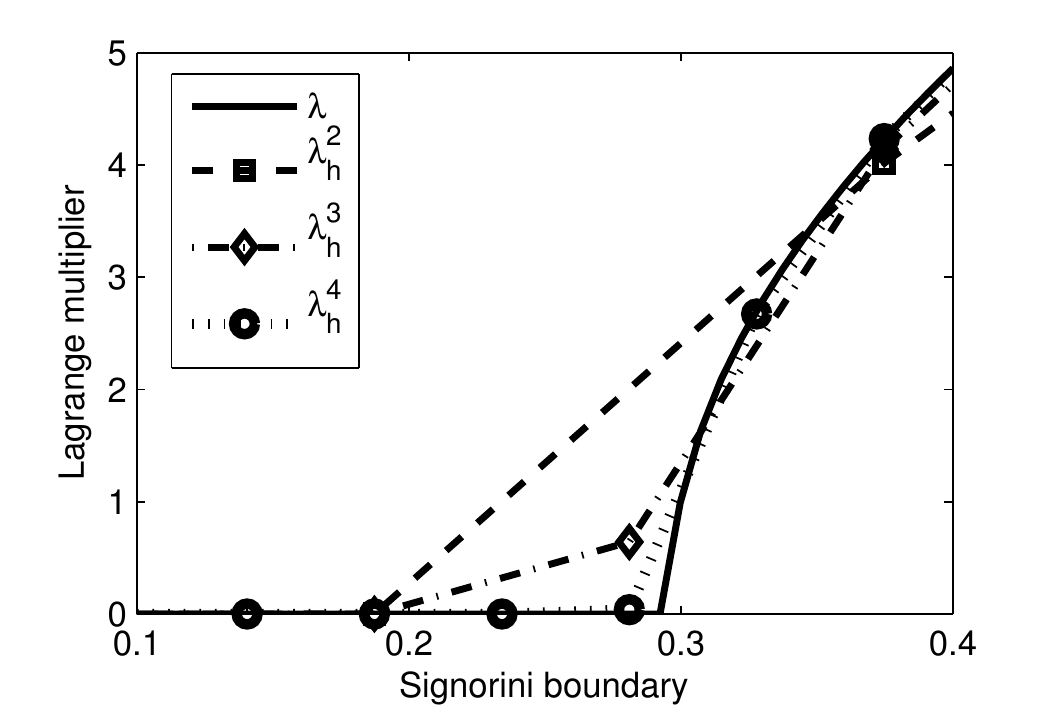}
\caption{  Zoom of dual solution and approximations at levels 2 to 4 around the left transmission point. }
\label{fig:transmission_pt}
\end{figure}

%% file: 08_conclusion.tex
In this work, we proved optimal order convergence in the $H^{1/2}(\Gamma_S)$ norm  for a standard finite element approximation of  Signorini problems. Based on this estimate, an optimal order error bound for the Lagrange multiplier, i.e., the flux, was derived in the $H^{-1/2}(\Gamma_S)$ norm and an improved bound for the primal error in the  $L^2(\Omega)$ norm was shown as a corollary.

Our analysis is based on a variational formulation of the continuous and the discrete Schur complement system which are variational inequalities posed over $\Gamma_S$. 
The difficulties arising from the nonlinearity could be handled by a Strang lemma, resulting in two terms. One term was a Galerkin discretization error on $\Gamma_S$, which could be bounded by standard techniques. To bound the second term, a trace error  of a linear problem posed on the whole domain, modern duality techniques with local estimates were adapted to the given situation.

A numerical example confirmed the optimal bounds and showed a good resolution of the active set. It also revealed that a gap remains between the theoretical and numerical results in the $L^2(\Omega)$ norm. As it was noted in Section~\ref{sec:L2_improved}, improved bounds in the  $L^2(\Gamma_S)$ norm would directly imply improved bounds in the $L^2(\Omega)$ norm.